\documentclass[11pt,letterpaper]{amsart}

\usepackage{amsmath,amssymb,amsfonts,amsthm,mathrsfs,mathtools}
\usepackage[margin=1.15in]{geometry}
\usepackage{hyperref}
\usepackage{enumitem}
\usepackage{microtype}

\newcommand{\Z}{\mathbb{Z}}
\newcommand{\Q}{\mathbb{Q}}
\newcommand{\R}{\mathbb{R}}

\newcommand{\F}{\mathbb{F}}

\newcommand{\Ocal}{\mathcal{O}}
\newcommand{\disc}{\operatorname{disc}}
\newcommand{\Li}{\operatorname{Li}}
\newcommand{\Res}{\operatorname{Res}}
\newcommand{\Gal}{\operatorname{Gal}}
\newcommand{\Frob}{\operatorname{Frob}}
\newcommand{\GL}{\operatorname{GL}}

\newcommand{\vvp}{\nu} % p-adic valuation (user may replace)

\newcommand{\A}{\mathbf{A}}
\newcommand{\Gm}{\mathbf{G}_m}
\newcommand{\Ga}{\mathbf{G}_a}
\newcommand{\Spec}{\mathrm{Spec}}
\newcommand{\Pic}{\mathrm{Pic}}

\newcommand{\wedgeTop}{\bigwedge\nolimits^{\mathrm{top}}}
\newcommand{\OO}{\mathcal{O}}
\newcommand{\cM}{\mathcal{M}}
\newcommand{\cN}{\mathcal{N}}
\newcommand{\Aff}{\mathrm{Aff}}

\theoremstyle{plain}
\newtheorem{theorem}{Theorem}[section]
\newtheorem{proposition}[theorem]{Proposition}
\newtheorem{lemma}[theorem]{Lemma}
\newtheorem{corollary}[theorem]{Corollary}

\theoremstyle{definition}
\newtheorem{definition}[theorem]{Definition}

\theoremstyle{remark}
\newtheorem{remark}[theorem]{Remark}
\newtheorem{example}[theorem]{Example}

\theoremstyle{plain}

\title{Eisenstein-prime Obstruction Sieve for Monogenicity}

\author[K.-H. Nguyen-Dang]{Khai-Hoan Nguyen-Dang}
\address{Morningside Center of Mathematics, Chinese Academy of Sciences, No.\ 55, Zhongguancun East Road, Beijing 100190, China}
\email{khaihoann@gmail.com}


\date{\today}
\begin{document}
\maketitle
\begin{abstract}
Alp\"oge--Bhargava--Shnidman showed that even a strengthened \emph{no local obstruction} condition for monogenicity does not force a global power integral basis: in the full spaces of cubic and quartic fields, a positive proportion are non-monogenic yet satisfy this ABS fixed-sign condition \cite{ABSCubic,ABSQuartic}. This raises a natural family-level question: \emph{does the same phenomenon persist inside one-parameter families, where the local structure varies in a highly constrained way?}

In this paper we answer this in the negative for the pure fields $K_m=\Q(\alpha)$ with $\alpha^n=m$ ($n\ge 4$) and $m$ square-free. Writing $g(m)=[\Ocal_{K_m}:\Z[\alpha]]$, we prove that the set of square-free $m$ for which $g(m)>1$ but $K_m$ has \emph{no} ABS local obstruction has natural density $0$. Consequently, in the pure family monogenicity and $\alpha$--monogenicity have the same natural density. The proof isolates a reusable mechanism, which we call the \emph{Eisenstein-prime obstruction sieve}. The argument is packaged in an abstract template and transfers to other Eisenstein parameter families.
\end{abstract}

\tableofcontents
%==========================================================
\section{Introduction}
\label{sec:introduction}
%==========================================================

\subsection{Monogenicity, index forms, and fixed-sign local solvability}

Let $K/\Q$ be a number field of degree $n\ge2$ with ring of integers $\Ocal_K$. The field $K$ is \emph{monogenic} if $\Ocal_K=\Z[\theta]$ for some $\theta\in\Ocal_K$, equivalently if $\Ocal_K$ admits a power integral basis $\{1,\theta,\dots,\theta^{n-1}\}$. The study of monogenicity is a classical problem in algebraic number theory, originally intertwined with the computation of rings of integers,
discriminants, integral bases and Diophantine equations (see, for example, \cite{EG17, Gaal2019Book,GaalSurvey}). A particularly rich testing ground is the family of \emph{pure number fields} $K_m=\Q(\sqrt[n]{m})$, where one can simultaneously pursue:
(i) explicit integral bases and index computations, and
(ii) global statistics and density questions as $m$ varies.
A long line of work develops explicit integral bases for pure fields and reveals striking periodic dependence on congruence classes of $m$ (see \cite{GaalRemete2017,JKS2021,ND25}).  These explicit bases typically show that the order $\Z[\sqrt[n]{m}]$ can fail to be maximal on congruence-defined subfamilies, thereby separating $\alpha$-monogenicity from monogenicity.

Recall that fixing an orientation $\omega\in\bigwedge_\Z^n\Ocal_K$ identifies the classical \emph{index form}
$f_{K,\omega}:\Ocal_K\to\Z$ of degree $N=\frac{n(n-1)}2$ via
\[
1\wedge \beta\wedge \beta^2\wedge\cdots\wedge \beta^{n-1}
=
f_{K,\omega}(\beta)\,\omega.
\]
Monogenicity is then equivalent to the Diophantine condition that $f_{K,\omega}$ represents $\pm1$ over $\Z$. From the perspective of modern arithmetic geometry, the index form is not merely a convenient polynomial: it is the coordinate expression of a determinant section cutting out monogenic generators on a canonically defined moduli space of generators (see \cite{ABHS-I,ABHS-II} and \S\ref{sec:ABHS}). This perspective cleanly packages local obstructions and clarifies how twisting and local conditions interact with the existence of monogenic generators.

In parallel, the modern arithmetic-statistics program has produced deep theorems on counting number fields and on the distribution of arithmetic properties in large families, culminating in Bhargava's parametrizations and their many extensions. Against this backdrop, the work of Alp\"oge--Bhargava--Shnidman shows that, in the \emph{full} families of cubic and quartic fields, there exist global failures of monogenicity that persist even when every local obstruction vanishes. This demonstrates that monogenicity is not governed purely by local conditions in these generic
families, and it motivates the question of whether there are natural families where a local-to-global principle might nevertheless hold in an asymptotic sense. More precisely, they introduced a reinforced notion in which the sign on the right-hand side is fixed \emph{globally} across all places (Definition~\ref{def:ABS}). They proved that this fixed-sign local condition does \emph{not} force global monogenicity: when ordered by discriminant, a positive proportion of cubic fields and a positive proportion of quartic fields are \emph{non}-monogenic yet satisfy the ABS fixed-sign local solvability condition \cite{ABSCubic,ABSQuartic}. This provides a robust failure of a local--global principle for monogenicity in high-dimensional parameter spaces. This paper investigates whether the ABS phenomenon persists inside \emph{Eisenstein families}, where the local structure is constrained by a distinguished generator and where a \emph{moving totally ramified prime} is built into the parameterization. 

More precisely, fix an integer $n\ge4$ and consider the pure fields
\[
K_m=\Q(\alpha),\qquad \alpha^n=m,
\]
with $m\in\Z\setminus\{0,\pm 1\}$ square-free, i.e., $X^n-m$ is irreducible over $\Q$. Write $\Ocal_m:=\Ocal_{K_m}$ and denote the \emph{index}
\[
g(m):=[\Ocal_m:\Z[\alpha]]\in\Z_{\ge1}.
\]
Thus $g(m)=1$ is the distinguished-generator notion of \emph{$\alpha$--monogenicity}, while \emph{monogenicity} means $\Ocal_m=\Z[\theta]$ for \emph{some} $\theta\in\Ocal_m$.

Even in the pure family, $\alpha$--monogenicity is strictly stronger than monogenicity in small degrees. For example, for $n=2$ one has $\Ocal_{\Q(\sqrt{m})}=\Z[(1+\sqrt{m})/2]$ when $m\equiv1\pmod4$, so $\Z[\sqrt{m}]$ fails to be maximal on a positive-density set of parameters, although quadratic fields are always monogenic \cite{NeukirchANT}. Similarly, in degree $n=3$ the order $\Z[\sqrt[3]{m}]$ fails to be maximal on explicit congruence classes
(e.g.\ $m\equiv \pm1\pmod9$ in the cube-free setting), yet the ambient cubic fields remain monogenic by explicit integral basis descriptions \cite{AyginNguyen2021}. Consequently, for $n=2,3$ the density of monogenicity differs from the density of $\alpha$--monogenicity in the pure family. Beyond very low degrees (even where explicit integral bases and index form computations are available), a general density comparison between monogenicity and $\alpha$--monogenicity for pure fields has not,
to our knowledge, been previously accessible. The guiding question is therefore:

\begin{quote}
\emph{In the pure family, can a positive-density set of non-$\alpha$--monogenic fields still satisfy the ABS fixed-sign \emph{no local obstruction} condition?  In other words, do monogenicity and $\alpha$--monogenicity differ in density for $n\ge4$?}
\end{quote}

Our main results answer this in the \emph{negative} for all $n\ge4$ and isolate a reusable mechanism which explains why.

\subsection{Main results}

Our first theorem shows that, in the pure family, the ABS fixed-sign local condition is \emph{generically}
incompatible with the failure of $\alpha$--monogenicity.

\begin{theorem}[Theorem~\ref{thm:density-zero-ABS} ]
\label{thm:intro-density-zero}
Fix $n\ge4$ and let $m$ range over square-free integers with $X^n-m$ irreducible over $\Q$. Then the set of parameters $m$ such that $g(m)\ge2$ but $K_m$ has \emph{no} ABS fixed-sign local obstruction has two-sided natural density $0$.
\end{theorem}

A conceptual consequence is a density-level identification of \emph{monogenicity} and
\emph{$\alpha$--monogenicity} in all degrees $n\ge4$.

\begin{theorem}[Theorem~\ref{thm:density-equality}]
\label{thm:intro-density-equality}
Fix $n\ge4$ and let $m$ range over square-free integers with $X^n-m$ irreducible.
Then the set of monogenic fields $K_m$ differs from the set of $\alpha$--monogenic fields $K_m$ by a density-zero subset of parameters. Equivalently, monogenicity and $\alpha$--monogenicity have the same two-sided natural density in the pure family.
\end{theorem}

In particular, combining the density equality with the explicit density of $\alpha$--monogenicity from the $\alpha$--criterion (see \cite{NDN25}) yields the explicit monogenic density:
\[
\delta\bigl(\{m:\ \Ocal_{K_m}\ \text{monogenic}\}\bigr)
=
\frac{6}{\pi^2}\prod_{p\mid n}\frac{p}{p+1}.
\]
This is conceptually surprising when viewed against the $n=2,3$ landscape: although the distinguished
generator is \emph{not} expected to capture all global monogenic generators in any fixed degree,
for $n\ge4$ it nevertheless captures them \emph{in density} inside the pure family. This sharp contrast suggests that the failure (or validity) of local-to-global principles for monogenicity is highly sensitive to the geometry of the family: generic families exhibit abundant global failures, whereas Eisenstein families exhibit a density-level local-to-global principle. We also prove quantitative refinements, including log-power savings at fixed index (see Proposition~\ref{prop:logpower-fixed-g}).

Moreover, we introduce a flexible framework of \emph{scaled Eisenstein} families in which the Eisenstein--prime obstruction sieve goes through \emph{verbatim}, once one has (a) uniform control of the relevant index targets and (b) nontriviality of the associated Kummer classes.

More precisely, fix a polynomial
\[
h(X)=c_{n-1}X^{n-1}+c_{n-2}X^{n-2}+\cdots+c_1X+c_0\in\Z[X],
\qquad c_0\neq 0.
\]
For each integer parameter $t\in\Z$, consider the monic polynomial
\[
f_t(X)\ :=\ X^n+t\,h(X)\ \in\ \Z[X],
\]
and let $\theta_t$ be a root.  Define
\[
K_t:=\Q(\theta_t),\qquad \Ocal_t:=\Ocal_{K_t},\qquad g(t):=[\Ocal_t:\Z[\theta_t]]\in\Z_{\ge 1}.
\]
We say that $K_t$ is \emph{$\theta_t$--monogenic} if $g(t)=1$. Consider the square-free parameter set
\[
\mathcal{T}_{h,\mathrm{sf}}
:=\Bigl\{\,t\in\Z:\ |t|>1,\ t\ \text{square-free, and }\gcd(t,c_0)=1\,\Bigr\}.
\]

We prove that, within the scaled Eisenstein family, the set of parameters that are \textsc{ABS}-unobstructed yet fail to be $\theta_t$--monogenic has density zero.

\begin{theorem}[Theorem~\ref{thm:scaled-density-zero-merged}]
Fix $n\ge 4$ and $h\in\Z[X]$ as above, and let $N=\frac{n(n-1)}2$.
Assume the following hypotheses on the square-free parameter set $\mathcal{T}_{h,\mathrm{sf}}$:
\begin{enumerate}[label=\textup{(\alph*)}]
\item \textup{(Finite index values)} There exists a finite set $G\subset\Z_{\ge1}$ such that
$g(t)\in G$ for all $t\in\mathcal{T}_{h,\mathrm{sf}}$.
\item \textup{(Kummer nontriviality)} For every $g\in G$ with $g\ge 2$,
\[
\Q(\zeta_{2N},g^{1/N})\neq \Q(\zeta_{2N}).
\]
\end{enumerate}
Let $\mathcal{S}'_{h}\subseteq\mathcal{T}_{h,\mathrm{sf}}$ be the set of $t$ such that:
\begin{enumerate}[label=\textup{(\roman*)}]
\item $g(t)\ge 2$, and
\item $K_t$ has \emph{no} local obstruction to monogenicity in the \textsc{ABS} fixed-sign sense.
\end{enumerate}
Then $\mathcal{S}'_{h}$ has two-sided natural density $0$ in $\Z$.
\end{theorem}

We exhibit explicit thin-parameter subfamilies showing that: (i) the Kummer nontriviality hypothesis is not a technical artifact but a genuinely necessary input for the sieve (Theorem~\ref{thm:thin-counterexample}); and (ii) even in one-parameter families, the \textsc{ABS} fixed--sign local condition does \emph{not} in general force the \emph{distinguished} generator to have index~$1$ (Theorem~\ref{thm:fixed-index-twist}), in line with the broader phenomenon that local conditions may fail to govern global monogenicity.

\subsection{Method: the Eisenstein-prime obstruction sieve}
The global argument is not a case-by-case integral-basis computation; rather, it is a general sieve strategy that turns local coset rigidity into density bounds by producing many primes $q$ at which $x^n-m$ is Eisenstein and then showing that a single global generator cannot satisfy all local coset constraints simultaneously except on a density-zero set.

The mechanism is deliberately modular and can be summarized as a three-step pipeline:

\smallskip\noindent
\textbf{(1) Local rigidity: a single-coset value set on local generators.}
Let $R$ be a complete DVR with residue characteristic $p\nmid N=\dfrac{n(n-1)}{2}$, and let $A/R$ be a rank-$n$ DVR of residue degree $1$ generated by a uniformizer.
Theorem~\ref{thm:localcoset} shows that for any orientation $\omega$ of $A$
the unit values of the local index form $f_\omega$ on \emph{all} $R$-generators of $A$ form exactly one coset:
\[
\{\,f_\omega(\beta)\in R^\times:\ A=R[\beta]\,\}
=
f_\omega(\pi)\cdot (R^\times)^N.
\]
Equivalently, the induced map to $H^1(\Spec R,\mu_N)\simeq R^\times/(R^\times)^N$ is constant and defines
a canonical \emph{local obstruction class} $\kappa(A/R,\omega)$.

\smallskip\noindent
\textbf{(2) One-prime obstruction certificates at moving Eisenstein primes.}
In the pure family, every prime $q\mid m$ is Eisenstein for $X^n-m$, so
$\Ocal_m\otimes\Z_q$ is precisely of the above type (a residue-degree $1$ DVR generated by $\alpha$). A top-wedge computation identifies $f_{m,\omega}(\alpha)=\pm g(m)$ (Lemma~\ref{lem:index-equals-indexform}),
and moreover $q\nmid g(m)$ for $q\mid m$ (Lemma~\ref{lem:eisenstein-prime-does-not-divide-g}),
so $g(m)$ is a unit in $\Z_q$.
Thus Theorem~\ref{thm:localcoset} converts fixed-sign local solvability at $q$ into the Kummer constraint
\[
\pm 1\in \pm g(m)\cdot(\Z_q^\times)^N.
\]
For primes $q\equiv 1\pmod{2N}$ the sign is itself an $N$th power, so solvability forces the residue-field condition
\[
\overline{g(m)}\in(\F_q^\times)^N.
\]
Then Kummer theory translates \(\,N\)-th power conditions into Frobenius conditions. More precisely, fix \(n\ge 4\) and \(N=\frac{n(n-1)}2\), and for a fixed index value \(g\ge 2\) put
\(K=\Q(\zeta_{2N})\) and \(L=K(g^{1/N})\).
For primes \(q\nmid 2Ng\) with \(q\equiv 1\pmod{2N}\) (so that \(q\) splits completely in \(K\)),
Kummer theory identifies the residue condition \(g\in (\F_q^\times)^N\) with a splitting condition in the
Kummer extension \(L/K\):
\[
q \text{ splits completely in }L
\quad\Longleftrightarrow\quad
g\in (\F_q^\times)^N.
\]
Equivalently, the condition \(g\) is an \(N\)-th power modulo \(q\) is encoded by the Artin/Frobenius element in \(\Gal(L/K)\) (or, in class field theoretic language, by the power residue symbol in a Kummer extension); see, for example, Milne's treatment of the Artin map and Kummer extensions \cite[\S IV.4--IV.5]{MilneCFT}. In other words, \emph{a single Eisenstein prime divisor of $m$} can certify the existence of an ABS fixed-sign local obstruction by a concrete $N$th-power residue test on $g(m)$.

\smallskip\noindent
\textbf{(3) Prime production and sieve closure.}
Assume \(L\neq K\).  Then Chebotarev implies that the set of primes that split completely in \(K\) but not in \(L\)
has positive density, and (away from the finitely many primes dividing \(2Ng\)) this is exactly the set
\[
P_g=\Bigl\{\,q:\ q\nmid 2Ng,\ q\equiv 1\!\!\pmod{2N},\ g\notin(\F_q^\times)^N\,\Bigr\}.
\]
Thus \(P_g\) is (up to finitely many primes) a Chebotarev set in \(L/\Q\), hence a frobenian set in the sense of
Serre \cite{SerreChebotarev}.  A classical sieve argument then bounds the set of integers whose prime factors avoid \(P_g\); see, e.g., \cite[Ch.~5]{HR74} or \cite[\S6.4]{IK04}.
This is the closure step used repeatedly in the paper: once \(P_g\) is produced with positive density, most parameters in a thick set (such as squarefree integers in a progression) necessarily have a
prime divisor in \(P_g\), which then forces a fixed-sign local obstruction at that prime via the one-prime
certificate discussed above. Since $\sum_{q\in P_g}1/q=\infty$, avoiding $P_g$ forces density zero, with quantitative log-power bounds
available via a Chebotarev--Mertens estimate and an upper-bound sieve. A sieve upper bound then shows that
parameters avoiding this Frobenian set are sparse (see \cite{HR74,AKK22}).

\subsection*{Remarks on the method}
We refer to this three-step pipeline as the \emph{Eisenstein-prime obstruction sieve}: the obstructing primes are rational primes \(q\) at which the specialization polynomial is Eisenstein
(typically \(q\mid\) the parameter), and the sieve is the classical sifting step excluding integers divisible by primes in the obstructing set \(P_g\). We emphasize that Eisenstein prime here means \emph{Eisenstein for the specialization polynomial} (not a prime in the ring of Eisenstein integers \(\Z[\omega]\)). 

This paradigm, Chebotarev sets (or frobenian sets) combined with a sieve argument, is classical and is treated systematically by Serre in his study of frobenian sets, where he explicitly combines Chebotarev-style arguments with Selberg-type sieve techniques \cite{SerreChebotarev}.  The new point in the present paper is that the \emph{local} obstruction is extracted in a particularly explicit form from primes dividing the parameter, via Eisenstein portability and our local coset constraint.

From the perspective of arithmetic statistics, our closure step is a one-dimensional analogue of the general principle of imposing finitely many local conditions via sieve methods.  In higher-dimensional parameter spaces, this is often carried out via the \emph{Ekedahl sieve} (originating in Ekedahl's work \cite{Ekedahl1991})
and its refinements (e.g.\ Poonen \cite{Poonen2003Squarefree}, and Bhargava's geometric sieve
framework \cite{Bhargava2014GeometricSieve}; see also the applications to squarefree discriminants and maximality problems in Bhargava--Shankar--Wang \cite{BSW-I,BSW-II}). In contrast, our distinctive contribution is the explicit \emph{one-prime} extraction of a fixed-sign local obstruction at Eisenstein prime divisors, which interfaces cleanly with a classical Chebotarev sieve on the parameter.

Our sieve also extends beyond the specific target $\pm1 $. The reason is conceptual: the local rigidity input concerns unit cosets modulo $N$-th powers, rather than the sign classes ${\pm1} $ per se. In particular, one can study, for example, the density of parameters for which there is local solvability of equations of the form
  \[
  f=\varepsilon\cdot u,
  \]
with $u$ ranging over prescribed unit classes; and in rigid families, the scarcity of parameters that are globally solvable yet locally everywhere solvable at a fixed index target.

A notable feature of the obstruction sieve is that the family-specific inputs are minimal: one needs (i) a supply of \emph{variable Eisenstein primes} indexed by the parameter, and
(ii) control of the set of possible indices $g(\cdot)$ on the parameter set.
This allows the method to transfer to other Eisenstein parameter families; see
\S\ref{sec:examples-counterexamples}, where we develop a general scaled Eisenstein framework and prove density-zero results under finite index-value and Kummer-nontriviality hypotheses (e.g.\ Theorem~\ref{thm:scaled-density-zero-merged}).

At the same time, the paper records constructions illustrating the necessity of the hypotheses: in thin parameter sets (e.g.\ primes) avoidance of a positive-density prime set may not be rare,
and if $g$ is an $N$th power in $\Q^\times$ then the relevant Kummer extension becomes trivial and the
Chebotarev obstruction set $P_g$ collapses (Remark~\ref{rem:kummer-trivial}). We also include a fixed-index twist construction (Theorem~\ref{thm:fixed-index-twist}) showing that,
in general one-parameter families, ABS-unobstructedness does not force a \emph{distinguished} generator to have index $1$, even on positive-density parameter sets.

\subsection{Related works and future directions}

Fix a finite set of places $S$. One may call an order \emph{$S$-monogenic} if
$\mathcal{O}_K[1/S]=\mathbb{Z}[1/S][\theta]$, i.e.\ monogenicity after allowing denominators
supported on $S$. In the generator-scheme formalism \cite{ABHS-I,ABHS-II}, this is naturally
encoded by replacing an affine generator space by a projective (or partially compactified) one:
informally, generators are allowed to acquire poles along $S$.
The obstruction sieve should extend verbatim after \emph{excluding} primes in $S$ from the portable set, leading to a density statement "away from $S$".

A sharper variant is to define \emph{projective monogenicity} up to scaling: ask whether $\mathcal{O}_K$ is generated by some $\theta$ modulo the $\mathbb{G}_m$--action (or more generally an affine group action), turning the obstruction class into a twist-class in
the relevant cohomology group. It is closely related to the notion of rational monogenicity \cite{Eve23, EG24}.
A worthwhile goal is to build the corresponding projective generator space and identify the correct
line bundle whose local value set admits a one-coset description.

Let $K/F$ be a finite extension of number fields. Relative monogenicity asks whether
$\mathcal{O}_K=\mathcal{O}_F[\theta]$. Many explicit computations especially in Kummer-type extensions, for instance \cite{SmithRadical2021}, indicate that residue-degree-one,
totally ramified primes of $F$ behave as the right portable primes, and the relevant obstruction
classes land in $\mathcal{O}_{F,\mathfrak{q}}^\times/(\mathcal{O}_{F,\mathfrak{q}}^\times)^N$.
This suggests a clean generalization of the sieve from $\mathbb{Q}$ to an arbitrary base field $F$.

Although different from power integral bases, the problem of freeness of $\mathcal{O}_L$ as a
$\mathbb{Z}[G]$--module exhibits a parallel structure: global freeness, local conditions, and
cohomological obstructions. For context on restricted Hilbert--Speiser/Leopoldt properties and
related obstructions, see \cite{byott2011hilbertspeiser} and the Hilbert--Speiser setting developed
around Swan modules \cite{greither1999swan}.
In parameterized Galois families with moving ramification, it is natural to ask whether one can
identify a \emph{portable local obstruction class} that produces Frobenian obstructing primes,
and then run the same Chebotarev--sieve closure.

\subsection*{Organization of the paper}

Section~\ref{sec:setting} recalls index forms and the ABHS generator scheme viewpoint.
Section~\ref{sec:kummer} develops the local obstruction class and proves local Kummer rigidity.
Section~\ref{sec:chebotarev-kummer} produces obstructing primes via Kummer theory and Chebotarev.
Sections~\ref{sec:density-pure-family} close the sieve and deduce the density-zero
and density-equality results in the pure family, including the explicit density formula.
Section~\ref{sec:quantitative-sieve} proves quantitative log-power bounds.
Finally, Section~\ref{sec:examples-counterexamples} develops transfer principles illustrating the scope of the method.

\subsection*{Acknowledgements}
We thank Leo Herr for useful communications. We thank Morningside Center of Mathematics, Chinese Academy of Sciences, for its support and a stimulating research environment.
% ============================================================
\section{Setting}\label{sec:setting}
% ============================================================

\subsection{Monogenicity and index forms}
Let $K/\Q$ be a number field of degree $n\ge 2$ with ring of integers $\OO_K$.
The ring $\OO_K$ is \emph{monogenic} if $\OO_K=\Z[\theta]$ for some $\theta\in\OO_K$. Equivalently, $\OO_K$ admits a power integral basis $\{1,\theta,\dots,\theta^{n-1}\}$.

Fix an \emph{orientation} $\omega\in \wedgeTop_\Z \OO_K$, i.e.\ a generator of the rank-one
$\Z$-module $\wedgeTop_\Z \OO_K$. Define the \emph{index form}
\begin{equation}\label{eq:index-form-classical}
f_{K,\omega}:\OO_K\longrightarrow \Z,\qquad
1\wedge \beta\wedge \beta^2\wedge \cdots \wedge \beta^{n-1} = f_{K,\omega}(\beta)\,\omega.
\end{equation}
Then $\OO_K$ is monogenic if and only if $f_{K,\omega}$ represents $\pm 1$ over $\Z$. 

Alp\"oge--Bhargava--Shnidman introduced a strengthened notion of \emph{no local obstruction} in which the sign on the right-hand side is fixed globally.

\begin{definition}\label{def:ABS}
Let $K/\Q$ be a number field of degree $n\ge 2$ with ring of integers $\OO_K$.
We say that $K$ is \emph{fixed-sign locally solvable for monogenicity} if there exist an orientation
$\omega$ of $\OO_K$ and $\varepsilon\in\{\pm1\}$ such that for every rational prime $\ell$ the equation
\[
f_{K,\omega}(x)=\varepsilon
\]
has a solution in $\OO_K\otimes_{\Z}\Z_\ell$.
\end{definition}

% ============================================================
\subsection{ABHS monogenerators and the index-form ideal}
\label{sec:ABHS}
% ============================================================

We recall the ABHS framework: monogenerators form a representable moduli scheme and the non-generator
locus is cut out by a canonically defined ideal (the index-form ideal) \cite{ABHS-I,ABHS-II}.

Let $\pi:S'\to S$ be finite locally free of constant degree $n\ge 1$ with $S$ locally noetherian.
A global section $\theta\in \Gamma(S',\OO_{S'})$ induces an $S$-morphism $S'\to \A^1_S$,
and $\theta$ is a \emph{monogenerator} iff this map is a closed immersion, equivalently iff
$\OO_{S'}=\OO_S[\theta]$ Zariski-locally on $S$.

\begin{theorem}[ABHS representability]\label{thm:ABHS-rep}
Let $\pi:S'\to S$ be finite locally free of constant degree $n$ with $S$ locally noetherian.
Let $X\to S$ be quasiprojective.
Then there exists a smooth, quasi-affine $S$-scheme $\cM_{X,S'/S}$ representing the functor on
$S$-schemes $T\to S$ that assigns the set of sections $S'\times_S T \to X\times_S T$
that are closed immersions over $T$.
If $X=\A^1_S$, then $\cM_{X,S'/S}$ is affine; write $\cM_{S'/S}:=\cM_{\A^1_S,S'/S}$.
\end{theorem}

Let $R_{S'/S}:=\Res_{S'/S}(\A^1_{S'})$ be the Weil restriction.
When $S=\Spec A$ and $S'=\Spec B$ with $B$ finite free of rank $n$ over $A$, one has
$R_{S'/S}\cong \A^n_S$ after choosing an $A$-basis of $B$.

Fix an $A$-basis $e_1,\dots,e_n$ of $B$ and let $x_1,\dots,x_n$ be the coordinates on $\A^n_S$.
Define the \emph{universal element}
\[
\Theta \;:=\; x_1 e_1 + \cdots + x_n e_n \;\in\; B\otimes_A A[x_1,\dots,x_n].
\]
For $1\le i,j\le n$ define $a_{ij}\in A[x_1,\dots,x_n]$ by
\[
\Theta^{i-1} \;=\; a_{i1}e_1 + \cdots + a_{in}e_n.
\]
Let $M(e_1,\dots,e_n):=(a_{ij})$ and define the \emph{local index form}
\[
i(e_1,\dots,e_n) \;:=\; \det(M(e_1,\dots,e_n)) \in A[x_1,\dots,x_n].
\]

\begin{theorem}[ABHS: the index form cuts out monogenerators]\label{thm:ABHS-index-open}
With notation as above, $\cM_{S'/S}$ identifies with the distinguished open
\[
\cM_{S'/S} \;\cong\; \Spec A[x_1,\dots,x_n, i(e_1,\dots,e_n)^{-1} ] \;\subseteq\; \Spec A[x_1,\dots,x_n].
\]
Moreover, under change of basis the polynomial $i(e_1,\dots,e_n)$ rescales by a unit, so the ideal it
generates glues to a global ideal sheaf $I_{S'/S}$ on $R_{S'/S}$.
Its vanishing defines a closed subscheme $\cN_{S'/S}\subseteq R_{S'/S}$, the \emph{scheme of non-monogenerators}.
\end{theorem}

Let $A$ be a commutative ring and $B$ a finite locally free $A$-algebra of rank $n$.
An \emph{orientation} of $B/A$ is a generator $\omega \in \wedgeTop_A B$.
Given $\omega$, define the \emph{index form} $f_\omega:B\to A$ by
\begin{equation}\label{eq:index-form-algebra}
1\wedge \beta \wedge \cdots \wedge \beta^{n-1} \;=\; f_\omega(\beta)\,\omega \qquad (\beta\in B).
\end{equation}

The following observations are easy.

\begin{lemma}[Translation invariance]\label{lem:translation}
Let $R$ be a commutative ring, $A$ a finite free $R$-algebra of rank $n$, $\beta\in A$, and $c\in R\subseteq A$.
Then in $\wedgeTop_R A$ one has
\[
1\wedge \beta\wedge\beta^2\wedge\cdots\wedge \beta^{n-1}
=
1\wedge (\beta-c)\wedge(\beta-c)^2\wedge\cdots\wedge (\beta-c)^{n-1}.
\]
In particular, $f_\omega(\beta)=f_\omega(\beta-c)$ for any orientation $\omega$.
\end{lemma}

\begin{proof}
For each $j\ge 0$, expand $\beta^j$ as an $R$-linear combination of $(\beta-c)^k$ with $k\le j$.
The transition matrix from $(1,\beta-c,\dots,(\beta-c)^{n-1})$ to $(1,\beta,\dots,\beta^{n-1})$
is upper triangular with diagonal entries $1$, so it has determinant $1$ and the top wedge is unchanged.
\end{proof}

\begin{lemma}[Unit wedge implies generation]\label{lem:unitwedge}
Let $R$ be a commutative ring and $A$ a finite free $R$-module of rank $n$.
Fix $\omega\in\wedgeTop_R A$.
If $\beta\in A$ satisfies $1\wedge\beta\wedge\cdots\wedge\beta^{n-1}=u\omega$ for some $u\in R^\times$,
then $1,\beta,\dots,\beta^{n-1}$ is an $R$-basis of $A$. In particular, if $A$ is an $R$-algebra then $A=R[\beta]$.
\end{lemma}

\begin{proof}
Choose an $R$-basis $e_1,\dots,e_n$ with $\omega=e_1\wedge\cdots\wedge e_n$.
The wedge equals $\det(T)\omega$ where $T$ is the coordinate matrix of $(1,\beta,\dots,\beta^{n-1})$.
If $\det(T)\in R^\times$, then $T\in \mathrm{GL}_n(R)$ and these elements form a basis.
\end{proof}

We can reinterprete the index form as follows.

\begin{lemma}[Determinant equals wedge coefficient]\label{lem:det-equals-wedge}
Assume $B$ is free over $A$ with basis $e_1,\dots,e_n$, and set $\omega:=e_1\wedge \cdots \wedge e_n$.
Let $\beta=\sum_{j=1}^n b_j e_j\in B$ and let $M(\beta)$ be the coefficient matrix of
$1,\beta,\dots,\beta^{n-1}$ relative to $e_1,\dots,e_n$.
Then $f_\omega(\beta)=\det(M(\beta))$.
In particular, under the universal element $\Theta=\sum x_j e_j$, one has
\[
i(e_1,\dots,e_n)= f_\omega(\Theta)\in A[x_1,\dots,x_n].
\]
\end{lemma}

\begin{proof}
Write $1,\beta,\dots,\beta^{n-1}$ in the basis $e_1,\dots,e_n$ to obtain the matrix $M(\beta)$.
Then multilinearity and alternation of $\wedge$ yield
\[
1\wedge\beta\wedge\cdots\wedge\beta^{n-1}
=\det(M(\beta))\,(e_1\wedge\cdots\wedge e_n),
\]
and comparison with \eqref{eq:index-form-algebra} gives the claim.
\end{proof}

\begin{remark}[Homogeneity]\label{rem:homog}
The polynomial $i(e_1,\dots,e_n)$ is homogeneous of total degree
\[
N=\sum_{i=1}^n(i-1)=n(n-1)/2
\]
in the variables $x_j$.
Equivalently, $f_\omega(u\beta)=u^N f_\omega(\beta)$ for $u\in A$.
\end{remark}

Throughout the paper, we set
\[
N \;:=\; \frac{n(n-1)}2.
\]
Throughout, $\mu_N$ denotes the fppf group scheme of $N$th roots of unity.
For a DVR $R$ we write $\mathfrak p$ for its maximal ideal and $k$ for its residue field. All schemes are assumed separated.

\begin{remark}
\label{rem:Ntame-vs-tame}
A prime $p$ is called \emph{$N$-tame} if $p\nmid N$. It is needed for the $N$th-power map $R^\times \xrightarrow{(\cdot)^N} R^\times$ to be \'etale and for the Hensel--Kummer arguments used below
(e.g.\ Lemma~\ref{lem:henselN}).

This condition is \emph{not} equivalent to classical tame ramification (i.e.\ $p\nmid n$).
Indeed, for every \emph{odd} prime $p$ one has
\[
p\nmid N \;\Longrightarrow\; p\nmid n,
\]
since $2$ is invertible modulo $p$ and $p\mid n$ would force $p\mid n(n-1)/2=N$.
However, at $p=2$ the implication can fail: if $n\equiv 2\pmod 4$ then $2\mid n$ but $2\nmid N$
(e.g.\ $n=6$, $N=15$).
\end{remark}

The following result says that index form value equals the global index up to sign.
\begin{lemma}
\label{lem:index-equals-indexform}
Let $K/\Q$ be a number field of degree $n\ge 2$ with ring of integers $\OO_K$.
Fix an orientation $\omega\in \wedge^n_{\Z}\OO_K$ and define the index form
$f_{K,\omega}:\OO_K\to\Z$. Let $\beta\in\OO_K$ be such that $K=\Q(\beta)$, and set
\[
I(\beta):=[\OO_K:\Z[\beta]]\in\Z_{\ge 1}.
\]
Then
\[
f_{K,\omega}(\beta)=\pm I(\beta),
\qquad\text{and hence}\qquad
[\OO_K:\Z[\beta]]=\bigl|f_{K,\omega}(\beta)\bigr|.
\]
In particular, $\OO_K=\Z[\beta]$ if and only if $f_{K,\omega}(\beta)=\pm 1$.
Moreover, replacing $\omega$ by $-\omega$ replaces $f_{K,\omega}$ by $-f_{K,\omega}$, so
$\bigl|f_{K,\omega}(\beta)\bigr|$ is independent of the choice of orientation.
\end{lemma}

\begin{proof}
Choose a $\Z$-basis $e_1,\dots,e_n$ of $\OO_K$, and write
$\omega_0:=e_1\wedge\cdots\wedge e_n$. Since $\wedge^n_\Z\OO_K$ is a free rank-$1$ $\Z$-module,
we have $\omega=\pm\omega_0$.

Because $K=\Q(\beta)$ and $\beta\in\OO_K$, the elements $1,\beta,\dots,\beta^{n-1}$ are
$\Q$-linearly independent and therefore $\Z$-linearly independent; hence they form a $\Z$-basis of
the order $\Z[\beta]$.  Write each $\beta^{j}$ in the basis $e_1,\dots,e_n$:
\[
\beta^{j}=\sum_{i=1}^n t_{i,j+1} e_i\qquad(0\le j\le n-1),
\]
and let $T=(t_{i,j})\in M_n(\Z)$ be the resulting change-of-basis matrix from
$(e_1,\dots,e_n)$ to $(1,\beta,\dots,\beta^{n-1})$.

By multilinearity and alternation of the wedge product,
\[
1\wedge \beta \wedge \cdots \wedge \beta^{n-1}
=
\det(T)\,(e_1\wedge\cdots\wedge e_n)
=
\det(T)\,\omega_0.
\]
Comparing with the defining relation
$1\wedge \beta \wedge \cdots \wedge \beta^{n-1}=f_{K,\omega}(\beta)\,\omega$
and using $\omega=\pm\omega_0$ gives $f_{K,\omega}(\beta)=\pm\det(T)$.

Finally, $\Z[\beta]\subseteq \OO_K$ is an inclusion of full-rank $\Z$-lattices whose matrix in the
bases $(1,\beta,\dots,\beta^{n-1})$ and $(e_1,\dots,e_n)$ is exactly $T$. Therefore
\[
[\OO_K:\Z[\beta]] = |\det(T)|,
\]
and combining with $f_{K,\omega}(\beta)=\pm\det(T)$ yields the claim.
\end{proof}

% ============================================================
\section{Kummer theory and the local obstruction class}
\label{sec:kummer}
% ============================================================

Let $S$ be a scheme on which $N$ is invertible. The Kummer exact sequence reads
\[
1\to \mu_N \to \Gm \xrightarrow{(\cdot)^N} \Gm \to 1.
\]
Taking cohomology yields an exact sequence
\begin{equation}\label{eq:kummer-H1}
0\to \Gamma(S,\OO_S^\times)/(\Gamma(S,\OO_S^\times))^N \longrightarrow H^1(S,\mu_N)
\longrightarrow \Pic(S)[N]\to 0.
\end{equation}

We have the following straightforward observation.

\begin{lemma}\label{lem:H1-DVR}
Let $R$ be a complete disrete valuation ring (DVR) such that $N$ is invertible in $R$.
Then $\Pic(\Spec R)=0$ and hence
\[
H^1(\Spec R,\mu_N) \;\cong\; R^\times/(R^\times)^N.
\]
\end{lemma}

\begin{proof}
A DVR is a PID, hence $\Pic(\Spec R)=0$. Apply \eqref{eq:kummer-H1}.
\end{proof}

\begin{definition}[Kummer class of a unit]\label{def:kummerclass}
Let $R$ be a ring with $N$ invertible.
For $u\in R^\times$, define its \emph{Kummer class} $[u]\in H^1(\Spec R,\mu_N)$ to be the class of the
$\mu_N$-torsor
\[
T_u := \Spec R[y]/(y^N-u).
\]
If $R$ is a DVR, $[u]$ corresponds to the coset $u\cdot (R^\times)^N$ via Lemma~\ref{lem:H1-DVR}.
\end{definition}

\begin{lemma}\label{lem:unitvalue}
Let $R$ be a complete DVR with $N$ invertible in $R$.
Let $A$ be a finite free commutative $R$-algebra of rank $n$ that is a DVR,
with residue degree $1$, and assume there exists a uniformizer $\pi\in A$ such that $A=R[\pi]$.
Fix an orientation $\omega\in \wedgeTop_R A$ and let $f_\omega:A\to R$ be \eqref{eq:index-form-algebra}. Then $f_\omega(\pi)\in R^\times$.
\end{lemma}

\begin{proof}
Since $A=R[\pi]$ and $A$ has rank $n$ over $R$, the set $\{1,\pi,\dots,\pi^{n-1}\}$ is an $R$-basis.
Thus $1\wedge\pi\wedge\cdots\wedge\pi^{n-1}$ is a generator of $\wedgeTop_R A$, hence equals a unit multiple of $\omega$.
\end{proof}

It gives rise to the following local obstruction class.

\begin{definition}\label{def:kappa}
Keep the notations as in~\ref{lem:unitvalue}. Define the \emph{local obstruction class}
\[
\kappa(A/R,\omega) \;:=\; [\,f_\omega(\pi)\,] \in H^1(\Spec R,\mu_N)\cong R^\times/(R^\times)^N.
\]
\end{definition}

To state our main result of this section, we need some elementary lemmas as follows.

\begin{lemma}\label{lem:henselN}
Let $R$ be a complete DVR with residue field $k$ of characteristic $p$.
Assume $p\nmid N$. Then for $u\in R^\times$,
\[
u\in (R^\times)^N \quad\Longleftrightarrow\quad \bar u \in (k^\times)^N.
\]
\end{lemma}

\begin{proof}
If $u=v^N$ then $\bar u=\bar v^{\,N}$.
Conversely, assume $\bar u=\bar v_0^{\,N}$ and lift $\bar v_0$ to $v_0\in R^\times$.
Let $F(X)=X^N-u$. Then $F(v_0)\equiv 0\pmod{\mathfrak p}$ and $F'(v_0)=Nv_0^{N-1}\in R^\times$
since $p\nmid N$. Hensel's lemma gives $v\in R^\times$ with $v^N=u$.
\end{proof}

\begin{lemma}\label{lem:e=n}
Let $R$ be a DVR with maximal ideal $\mathfrak p$, and let $A$ be a finite free $R$-algebra of rank $n$
that is a DVR with maximal ideal $\mathfrak m$ and residue field $A/\mathfrak m\cong R/\mathfrak p$.
Then the ramification index of $A/R$ is $e=n$, and $\mathfrak p A=\mathfrak m^n$.
In particular, if $n\ge 2$ then $\mathfrak pA\subseteq \mathfrak m^2$.
\end{lemma}

\begin{proof}
Since $A$ is a DVR finite over $R$, $A$ is torsion-free and hence finite flat over $R$ of rank $n$.
Let $e$ be the ramification index and $f$ the residue degree. Then $n=ef$.
The residue field assumption gives $f=1$, hence $e=n$.
The identity $\mathfrak pA=\mathfrak m^e$ holds for extensions of DVRs, giving $\mathfrak pA=\mathfrak m^n$.
\end{proof}

We state and prove our local Kummer rigidity.
\begin{theorem}\label{thm:localcoset}
Let $R$ be a complete DVR with maximal ideal $\mathfrak p$, residue field $k$, and residue characteristic $p$.
Let $A$ be a finite free commutative $R$-algebra of rank $n\ge 2$ such that:
\begin{enumerate}[label=\textup{(\roman*)},leftmargin=*]
\item $A$ is a DVR with maximal ideal $\mathfrak m$;
\item $A/\mathfrak m \cong k$;
\item there exists a uniformizer $\pi\in\mathfrak m$ with $A=R[\pi]$.
\end{enumerate}
Fix an orientation $\omega\in \wedgeTop_R A$ and let $f_\omega:A\to R$ be defined by \eqref{eq:index-form-algebra}.
Assume $p\nmid N$.

Then the set of unit values of $f_\omega$ on $R$-generators of $A$ is a single coset:
\[
\{\, f_\omega(\beta)\in R^\times : A=R[\beta] \,\}
=
f_\omega(\pi)\cdot (R^\times)^N.
\]
Equivalently, the composite map
\[
\cM_{A/R}(R)\xrightarrow{f_\omega} R^\times \longrightarrow R^\times/(R^\times)^N \;\cong\; H^1(\Spec R,\mu_N)
\]
is constant, with value $[f_\omega(\pi)]$.
\end{theorem}

\begin{proof}
Let $\beta\in A$ be an $R$-generator, so $A=R[\beta]$.
Choose $b_0\in R$ lifting the residue class of $\beta$ in $A/\mathfrak m\cong k$, and set $\beta_0:=\beta-b_0\in\mathfrak m$.
By Lemma~\ref{lem:translation}, $f_\omega(\beta)=f_\omega(\beta_0)$ and $A=R[\beta_0]$.

If $\beta_0\in \mathfrak m^2$, then for every polynomial $P(T)\in R[T]$ one has
$P(\beta_0)\equiv P(0)\pmod{\mathfrak m^2}$, hence $R[\beta_0]\subseteq R+\mathfrak m^2$.
We claim $R+\mathfrak m^2$ contains no element of valuation $1$ in $A$, contradicting $A=R[\beta_0]$.

Indeed, let $x=r+y$ with $r\in R$ and $y\in \mathfrak m^2$.
If $r\notin\mathfrak p$, then $r$ is a unit of $R$ and hence of $A$, so $v_A(x)=0$.
If $r\in\mathfrak p$, then $x\in \mathfrak pA+\mathfrak m^2$.
By Lemma~\ref{lem:e=n}, $\mathfrak pA=\mathfrak m^n\subseteq\mathfrak m^2$ (since $n\ge 2$),
so $x\in\mathfrak m^2$ and thus $v_A(x)\ge 2$.
In either case $v_A(x)\neq 1$, proving the claim.

Therefore $\beta_0\notin \mathfrak m^2$ and hence $v_A(\beta_0)=1$, i.e.\ $\beta_0$ is a uniformizer.

Thus $\beta_0=\pi u$ for some $u\in A^\times$.
Let $u_0\in k^\times$ be the image of $u$ modulo $\mathfrak m$.

Write $u=u_0+\pi w$ for some $w\in A$ (since $u\in A^\times$ and $u_0\in k^\times$ is its residue class).
Then for each $j\ge 1$ we have the identity in $A$
\[
(\pi u)^j=\pi^j(u_0+\pi w)^j=\pi^j u_0^j+\pi^{j+1}\cdot(\text{some element of }A),
\]
hence
\[
(\pi u)^j \equiv u_0^j\,\pi^j \pmod{\pi^{j+1}A}.
\]
Reducing modulo $\mathfrak pA=\mathfrak m^n$ and using $j+1\le n$ (since $1\le j\le n-1$) yields the
congruence
\[
\overline{(\pi u)^j}\equiv u_0^j\,\bar\pi^j \pmod{\bar\pi^{j+1}}
\qquad\text{in }A/\mathfrak pA.
\]
In $A/\mathfrak pA$, the images of $1,\pi,\dots,\pi^{n-1}$ form a $k$-basis and $\bar\pi$ is nilpotent of order $n$.
For $1\le j\le n-1$,
\[
\overline{\beta_0^{\,j}}
=\overline{(\pi u)^j}
=\bar\pi^{\,j}\bar u^{\,j}
\equiv u_0^j \bar\pi^{\,j}\pmod{\bar\pi^{\,j+1}},
\]
since $\bar u\equiv u_0\pmod{\bar\pi}$.
Hence the change-of-basis matrix $\bar T$ from $(1,\bar\pi,\dots,\bar\pi^{n-1})$ to
$(1,\overline{\beta_0},\dots,\overline{\beta_0^{n-1}})$ is upper triangular with diagonal entries
$1,u_0,u_0^2,\dots,u_0^{n-1}$, so
\[
\det(\bar T)=u_0^{1+2+\cdots+(n-1)}=u_0^N\in (k^\times)^N.
\]

Since $A=R[\beta_0]$ and $\beta_0$ is integral over $R$, the elements $1,\beta_0,\dots,\beta_0^{n-1}$ generate the free $R$-module $A$; as $R$ is local and $\mathrm{rank}_R(A)=n$, they form an $R$-basis, so the change-of-basis matrix $T$ indeed lies in $\GL_n(R)$.

Let $T\in\mathrm{GL}_n(R)$ be the change-of-basis matrix from $(1,\pi,\dots,\pi^{n-1})$ to $(1,\beta_0,\dots,\beta_0^{n-1})$.
Then
\[
1\wedge \beta_0\wedge\cdots\wedge \beta_0^{n-1} = \det(T)\cdot (1\wedge \pi\wedge\cdots\wedge\pi^{n-1}),
\]
so $f_\omega(\beta_0)=\det(T)f_\omega(\pi)$.
Reduction modulo $\mathfrak p$ gives $\det(T)\bmod\mathfrak p=\det(\bar T)\in (k^\times)^N$.
By Lemma~\ref{lem:henselN}, $\det(T)\in (R^\times)^N$.
Hence $f_\omega(\beta)=f_\omega(\beta_0)\in f_\omega(\pi)\cdot (R^\times)^N$.

For any $v\in R^\times$, set $\beta=v\pi$. Then $A=R[\beta]$ and by homogeneity
$f_\omega(\beta)=v^N f_\omega(\pi)$ (Remark~\ref{rem:homog}).
Thus all of $f_\omega(\pi)(R^\times)^N$ occurs.
\end{proof}

\begin{corollary}\label{cor:local-solv-kummer}
In the situation of Theorem~\ref{thm:localcoset}, for any $\varepsilon\in R^\times$ the equation
\[
f_\omega(x)=\varepsilon \qquad (x\in A)
\]
has a solution with $x$ an $R$-generator of $A$ if and only if
\[
[\varepsilon]=[f_\omega(\pi)] \quad \text{in } H^1(\Spec R,\mu_N)\cong R^\times/(R^\times)^N.
\]
Equivalently,
\[
f_\omega(\{\beta\in A:\ A=R[\beta]\})\cap R^\times = f_\omega(\pi)\cdot (R^\times)^N.
\]
\end{corollary}

\begin{proof}
By Theorem~\ref{thm:localcoset}, unit values on generators form the single coset $f_\omega(\pi)(R^\times)^N$.
The Kummer identification $R^\times/(R^\times)^N\simeq H^1(\Spec R,\mu_N)$ is Lemma~\ref{lem:H1-DVR}.
\end{proof}

\begin{remark}[Affine group viewpoint]
Translation invariance and homogeneity encode the action of $\Aff_1=\Ga\rtimes\Gm$ on $\A^1$.
Theorem~\ref{thm:localcoset} rigidifies the induced orbit geometry of $\cM_{A/R}(R)$ in the residue-degree~$1$ DVR case.
\end{remark}

%==========================================================
\section{Chebotarev--Kummer production of obstructing primes}
\label{sec:chebotarev-kummer}
%==========================================================

In this section we package the prime-production step used later in the sieve. Fix an integer $N\ge 2$ and $g\in\Z\setminus\{0\}$.  The condition
$g\in(\F_q^\times)^N$ (or not) is detected by splitting in a Kummer extension of a
cyclotomic field.  Chebotarev then supplies infinitely many primes $q$ with
$g\notin(\F_q^\times)^N$, and in fact a positive-density set.

%----------------------------------------------------------
\subsection{The Kummer--cyclotomic field and a splitting criterion}
%----------------------------------------------------------
Firstly, we recall some basics on the Galois theory of Kummer--cyclotomic fields.

\begin{lemma}
\label{lem:kummer-galois}
Let $N\ge 2$ and $g\in\Z\setminus\{0\}$. Put
\[
K:=\Q(\zeta_{2N}),\qquad E:=\Q(\zeta_{N},g^{1/N}),\qquad L:=K(g^{1/N})=\Q(\zeta_{2N},g^{1/N}).
\]
Then $K/\Q$ and $E/\Q$ are finite Galois extensions and $L=KE$.
In particular, $L/\Q$ is finite Galois, and $L$ is the splitting field over $\Q$
of the polynomial $(X^N-g)(X^{2N}-1)$.
\end{lemma}

\begin{proof}
The field $K=\Q(\zeta_{2N})$ is cyclotomic, hence finite Galois over $\Q$.
The field $E=\Q(\zeta_N,g^{1/N})$ is the splitting field of $X^N-g$ over $\Q$, hence finite Galois over $\Q$.
Since $\zeta_N\in K$ and $g^{1/N}\in E$, we have $L=\Q(\zeta_{2N},g^{1/N})=KE$.
As a compositum of finite Galois extensions, $L/\Q$ is finite Galois.
Finally, the compositum of splitting fields is the splitting field of the product polynomial.
\end{proof}

\begin{lemma}[Kummer splitting criterion in residue fields]
\label{lem:kummer-splitting-general}
Let $N\ge 2$ and $g\in\Z\setminus\{0\}$, and put $K=\Q(\zeta_{2N})$ and $L=K(g^{1/N})$.
Let $q$ be a rational prime satisfying $q\nmid 2Ng$ and $q\equiv 1\pmod{2N}$.
Then $q$ splits completely in $L$ if and only if $g\in(\F_q^\times)^N$.
\end{lemma}

\begin{proof}
Since $q\equiv 1\pmod{2N}$ and $q\nmid 2N$, the prime $q$ splits completely in $K=\Q(\zeta_{2N})$.
Thus for any prime $\mathfrak q\mid q$ of $\Ocal_K$ we have $\Ocal_K/\mathfrak q\simeq\F_q$.
Moreover, $\mu_N\subseteq K$ (because $\zeta_N=\zeta_{2N}^2\in K$), so $L/K$ is a (possibly proper) Kummer extension.

Let $\gamma:=g^{1/N}\in L$, so $\gamma^N=g\in\Z$.  Since $q\nmid Ng$, the prime $\mathfrak q$ is unramified in $L/K$.
Fix a prime $\mathfrak Q\mid\mathfrak q$ in $L$ and let $\Frob_{\mathfrak q}\in\Gal(L/K)$ be a Frobenius element.
For the algebraic integer $\gamma$ we have the standard congruence
\[
\Frob_{\mathfrak q}(\gamma)\equiv \gamma^{q}\pmod{\mathfrak Q}.
\]
Thus $\Frob_{\mathfrak q}=1$ if and only if $\gamma^{q-1}\equiv 1\pmod{\mathfrak Q}$.
Because $q\equiv 1\pmod{2N}$ we have $N\mid(q-1)$, so
\[
\gamma^{q-1}=\gamma^{N\cdot\frac{q-1}{N}}=(\gamma^N)^{\frac{q-1}{N}}=g^{\frac{q-1}{N}}.
\]
Reducing $g$ modulo $\mathfrak q$ (equivalently modulo $q$) gives an element of $\F_q^\times$.
In the cyclic group $\F_q^\times$ of order $q-1$ with $N\mid(q-1)$, the congruence
$g^{(q-1)/N}\equiv 1$ is equivalent to $g\in(\F_q^\times)^N$.
Hence $\Frob_{\mathfrak q}=1$ if and only if $g\in(\F_q^\times)^N$.

Finally, since $q$ splits completely in $K$, the prime $q$ splits completely in $L$ if and only if
$\mathfrak q$ splits completely in $L/K$, i.e.\ if and only if $\Frob_{\mathfrak q}=1$.
\end{proof}

% --- Chebotarev, in the form used for “splitting conditions give positive-density prime sets” ---
We recall the Chebotarev density theorem in \cite[§3]{SerreChebotarev}.
\begin{theorem}
Let $E/F$ be a finite Galois extension of number fields, with Galois group
$G=\mathrm{Gal}(E/F)$.  Let $C\subseteq G$ be a conjugacy-stable subset
(equivalently, a union of conjugacy classes).  For each nonzero prime ideal
$\mathfrak p\subset\mathcal O_F$ unramified in $E$, write
$\mathrm{Frob}_{\mathfrak p}(E/F)\subseteq G$ for its Frobenius conjugacy class.
Define
\[
\pi_C(x;E/F)
:=\#\Bigl\{\mathfrak p \subset \mathcal O_F :
\mathfrak p \text{ unramified in }E,\;
N\mathfrak p\le x,\;
\mathrm{Frob}_{\mathfrak p}(E/F)\subseteq C
\Bigr\}.
\]
Then, as $x\to\infty$,
\[
\pi_C(x;E/F)\sim \frac{|C|}{|G|}\,\mathrm{Li}(x),
\qquad\text{where }\mathrm{Li}(x):=\int_2^x \frac{dt}{\log t}.
\]
Equivalently, the set of (unramified) prime ideals $\mathfrak p$ of $F$ with
$\mathrm{Frob}_{\mathfrak p}(E/F)\subseteq C$ has natural density $|C|/|G|$.
\end{theorem}

The theorem yields the following important consequence.
\begin{corollary}
Let $E/F$ be as above.  The set of prime ideals of $F$ splitting completely in $E$
(i.e. $\mathrm{Frob}_{\mathfrak p}(E/F)=\{1\}$) has density $1/|G|=1/[E:F]$.

In particular, if $K\subsetneq L$ are finite Galois extensions of $\mathbb Q$, then
the set of rational primes $p$ that split completely in $K$ but do \emph{not} split
completely in $L$ has density
\[
\frac{1}{[K:\mathbb Q]}-\frac{1}{[L:\mathbb Q]}\;>\;0.
\]
\end{corollary}

We define Chebotarev--Kummer obstructing primes as follows.

\begin{definition}
\label{def:Pg-general}
Let $N\ge 2$ and $g\in\Z\setminus\{0\}$.
Assume $L=K(g^{1/N})\neq K$ for $K=\Q(\zeta_{2N})$.
Define
\[
P_g
:=\Bigl\{\,q\ \text{prime}:\ q\nmid 2Ng,\ q\equiv 1\!\!\!\pmod{2N},\ \text{and } g\notin(\F_q^\times)^N \Bigr\}.
\]
\end{definition}

Applying Chebotarev's theorem to obstructing primes yields the following.
\begin{proposition}
\label{prop:Pg-positive-density-general}
Assume $L\neq K$.
Then $P_g$ has (natural) density
\[
\delta_g=\frac{1}{[K:\Q]}-\frac{1}{[L:\Q]}>0
\]
among rational primes.  In particular, $P_g$ is infinite and
\[
\sum_{q\in P_g}\frac{1}{q}=\infty.
\]
\end{proposition}

\begin{proof}
By Lemma~\ref{lem:kummer-galois}, both $K/\Q$ and $L/\Q$ are finite Galois extensions with $K\subsetneq L$.
Chebotarev's density theorem implies that the set of rational primes splitting completely in a finite Galois
extension $F/\Q$ has density $1/[F:\Q]$.
Therefore, the set of primes splitting completely in $K$ but \emph{not} splitting completely in $L$ has density
\[
\frac{1}{[K:\Q]}-\frac{1}{[L:\Q]}>0.
\]
Discarding the finite set of primes dividing $2Ng$ does not change density.

If $q\nmid 2Ng$ splits completely in $K$, then $q\equiv 1\pmod{2N}$; moreover, by
Lemma~\ref{lem:kummer-splitting-general}, such a prime splits completely in $L$ if and only if
$g\in(\F_q^\times)^N$. Hence, up to finitely many exceptions, the condition
split completely in $K$ but not in $L$ is equivalent to $q\in P_g$.
This proves that $P_g$ has density $\delta_g$.

Finally, Chebotarev's theorem yields
$\#\{q\le X: q\in P_g\}\sim \delta_g\,X/\log X$, and partial summation gives
$\sum_{q\le X,\ q\in P_g}1/q=\delta_g\log\log X+O(1)$, implying $\sum_{q\in P_g}1/q=\infty$.
\end{proof}

\begin{corollary}[Sign-killing in $\F_q^\times$ and $\Z_q^\times$]
\label{cor:sign-killed}
Let $N\ge 1$ and let $q$ be a prime with $q\equiv 1\pmod{2N}$. Then
\[
-1\in(\F_q^\times)^N
\qquad\text{and}\qquad
-1\in(\Z_q^\times)^N.
\]
In particular,
\[
\pm(\F_q^\times)^N=(\F_q^\times)^N
\qquad\text{and}\qquad
\pm(\Z_q^\times)^N=(\Z_q^\times)^N.
\]
\end{corollary}

\begin{proof}
Since $q\equiv 1\pmod{2N}$, the cyclic group $\F_q^\times$ has an element of order $2N$; its $N$th power is $-1$,
so $-1\in(\F_q^\times)^N$.
Also $q\equiv 1\pmod{2N}$ implies $q>2N\ge N$, hence $q\nmid N$.
By Lemma~\ref{lem:henselN} (Hensel criterion for $N$th powers), the inclusion
$-1\in(\F_q^\times)^N$ lifts to $-1\in(\Z_q^\times)^N$.
\end{proof}

We recall the two-sided natural density.
\begin{definition}
\label{def:two-sided-density}
For a set $A\subseteq \Z$, its \emph{two-sided natural density} (if it exists) is
\[
\delta(A):=\lim_{X\to\infty}\frac{\#\{m\in A:\ |m|\le X\}}{2X}.
\]
\end{definition}

\begin{lemma}[Avoiding a divergent set of primes has density zero]
\label{lem:avoid-divergent-P}
Let $P$ be a set of rational primes such that $\sum_{p\in P}1/p=\infty$.
Let $A_P\subseteq \Z$ be the set of \emph{$P$-free} integers (i.e.\ $p\nmid m$ for all $p\in P$).
Then $\delta(A_P)=0$.
\end{lemma}

\begin{proof}
For $Y\ge 2$, let $P(Y)=\{p\in P:\ p\le Y\}$.
By inclusion--exclusion (or CRT), the set of integers divisible by no prime in $P(Y)$ has two-sided density
\[
\prod_{p\in P(Y)}\Bigl(1-\frac{1}{p}\Bigr).
\]
Hence the upper two-sided density of $A_P$ satisfies
\[
\overline{\delta}(A_P)\le \prod_{p\in P(Y)}\Bigl(1-\frac{1}{p}\Bigr)\qquad(Y\ge 2).
\]
Using $\log(1-x)\le -x$ for $0<x<1$ gives
\[
\log\prod_{p\in P(Y)}\Bigl(1-\frac{1}{p}\Bigr)\le -\sum_{p\in P(Y)}\frac{1}{p}\xrightarrow[Y\to\infty]{}-\infty,
\]
so the product tends to $0$ and therefore $\delta(A_P)=0$.
\end{proof}

\begin{remark}[Thin parameter sets]
Lemma~\ref{lem:avoid-divergent-P} concerns density in $\Z$ (and hence in any positive-density subfamily of $\Z$,
such as square-free integers). It does not automatically imply analogous statements in very thin parameter sets
(e.g.\ the primes), where a separate sieve input is required.
\end{remark}

%==========================================================
\subsection{Nontriviality of the Kummer extension for pure-field index values}
\label{sec:kummer-nontrivial-index}
%==========================================================

For the sieve arguments below we require that $L\neq K$ when $g=g(m)\ge 2$
arises as a pure power-basis index. It is noticed that Eisenstein primes do not divide the pure power-basis index.

\begin{lemma}
\label{lem:eisenstein-prime-does-not-divide-g}
Fix $n\ge 2$ and let $m\in\Z$ be square-free such that $X^n-m$ is irreducible over $\Q$.
Let $\alpha^n=m$, set $K_m=\Q(\alpha)$, and write
\[
g(m):=[\OO_{K_m}:\Z[\alpha]]\in\Z_{\ge 1}.
\]
If $q$ is a prime with $q\mid m$, then $X^n-m$ is Eisenstein at $q$, so
\[
\OO_{K_m}\otimes_\Z \Z_q \;=\; \Z_q[\alpha].
\]
Consequently,
\[
v_q\bigl(g(m)\bigr)=0
\quad\text{(equivalently, $q\nmid g(m)$),}
\]
so in particular $g(m)\in\Z_q^\times$ and its reduction $\overline{g(m)}\in\F_q^\times$ is well-defined.
\end{lemma}

\begin{proof}
Since $m$ is square-free and $q\mid m$, we have $v_q(m)=1$, and thus $X^n-m$ is Eisenstein at $q$.
Hence $\Z_q[\alpha]$ is a DVR and is integrally closed; it is therefore the full ring of integers of
$K_m\otimes_\Q \Q_q$. On the other hand, $\OO_{K_m}\otimes_\Z\Z_q$ is by definition the integral closure
of $\Z_q$ in $K_m\otimes_\Q\Q_q$. Thus $\OO_{K_m}\otimes_\Z\Z_q=\Z_q[\alpha]$.

Now localizing the index $g(m)=[\OO_{K_m}:\Z[\alpha]]$ at $q$ gives
\[
\bigl(\OO_{K_m}\otimes\Z_q : \Z_q[\alpha]\bigr)=q^{v_q(g(m))}.
\]
Since the two $\Z_q$-orders coincide, this local index is $1$, so $v_q(g(m))=0$ as claimed.
\end{proof}

Consequently, we can show the finiteness and exponent bounds for the pure power-basis index.
\begin{corollary}
\label{cor:finite-g}
Fix $n\ge 2$ and let $m$ be square-free such that $X^n-m$ is irreducible over $\Q$.
Let $\alpha^n=m$, put $K_m=\Q(\alpha)$, and write
\[
g(m):=[\Ocal_{K_m}:\Z[\alpha]].
\]
Then:
\begin{enumerate}[label=\textup{(\roman*)},leftmargin=*]
\item $\gcd(g(m),m)=1$.
\item One has the divisibility
\[
g(m)^2 \mid n^n.
\]
\item In particular, if $p\mid g(m)$ then $p\mid n$, and for each prime $p\mid n$,
\[
v_p(g(m))\le \frac{n}{2}\,v_p(n).
\]
Consequently, the set $\{g(m): m\ \text{square-free and }X^n-m\text{ irreducible}\}$ is finite.
\end{enumerate}
\end{corollary}

\begin{proof}
The discriminant of the order $\Z[\alpha]$ equals the polynomial discriminant:
\[
\disc(\Z[\alpha])=\disc(X^n-m)=(-1)^{\frac{n(n-1)}2}\,n^n\,m^{n-1}.
\]
The discriminant--index identity gives
\[
\disc(\Z[\alpha]) \;=\; g(m)^2\,\disc(\Ocal_{K_m}).
\]
In particular, $g(m)^2\mid \disc(\Z[\alpha])$, so
\[
g(m)^2 \mid n^n\,m^{n-1}.
\]

Now let $q\mid m$. Since $m$ is square-free, $v_q(m)=1$, hence $X^n-m$ is Eisenstein at $q$.
Therefore $\Ocal_{K_m}\otimes\Z_q=\Z_q[\alpha]$, so the $q$-part of the index is trivial and $q\nmid g(m)$.
As this holds for every prime $q\mid m$, we get $\gcd(g(m),m)=1$.

Combining $\gcd(g(m),m)=1$ with $g(m)^2\mid n^n m^{n-1}$ yields $g(m)^2\mid n^n$, proving (ii).
The valuation bound (iii) follows from $2v_p(g(m))\le v_p(n^n)=n\,v_p(n)$.
Finally, (ii)--(iii) imply $g(m)$ ranges over a finite set.
\end{proof}

\begin{lemma}[A growth inequality for $v_p(n)$]
\label{lem:vp-ineq}
Let $n\ge 5$ and let $p$ be a prime divisor of $n$. Then
\[
v_p(n)<\frac{n-1}{2}.
\]
\end{lemma}

\begin{proof}
If $n\in\{5,6\}$ then $v_p(n)=1<\frac{n-1}{2}$ for all $p\mid n$.
Assume $n\ge 7$. Then $v_p(n)\le \log_2(n)$ for all primes $p\mid n$.
We claim $\log_2(n)<\frac{n-1}{2}$ for $n\ge 7$.
Indeed, $2^{(n-1)/2}>n$ holds for $n=7,8$ by inspection, and if it holds for some $n\ge 7$ then
\[
2^{(n+1)/2}=2\cdot 2^{(n-1)/2}>2n\ge n+2,
\]
so it holds inductively for all larger $n$. Hence $\log_2(n)<(n-1)/2$.
\end{proof}

We confirm the nontriviality of the Kummer extension for pure index values in the following result.
\begin{lemma}
\label{lem:kummer-nontrivial}
Let $n\ge 4$ and $N=\frac{n(n-1)}2$.
Let $g\ge 2$ be an integer occurring as $g=g(m)$ for some square-free $m$ with $X^n-m$ irreducible.
Let $K:=\Q(\zeta_{2N})$ and $L:=K(g^{1/N})$.
Then $L\neq K$.
\end{lemma}

\begin{proof}
Assume for contradiction that $L=K$, i.e.\ $g^{1/N}\in K$.
Then $\Q(g^{1/N})\subseteq K$.
Since $K/\Q$ is cyclotomic, it is abelian; hence every intermediate field of $K/\Q$ is Galois over $\Q$.
In particular, $\Q(g^{1/N})/\Q$ is Galois.

Write $g=\prod_\ell \ell^{e_\ell}$ and set $d:=\gcd\{e_\ell:\ e_\ell>0\}$, so $g=h^d$ with $h\in\Z_{\ge 2}$
not an $r$th power in $\Q$ for any prime $r$.
Let $d_0=\gcd(N,d)$ and define $a=d/d_0$, $b=N/d_0$ so that $\gcd(a,b)=1$ and
\[
g^{1/N}=(h^a)^{1/b}.
\]
By the binomial irreducibility criterion (and the fact $\gcd(a,b)=1$), the polynomial $X^b-h^a$ is irreducible over $\Q$.
Hence $[\Q(g^{1/N}):\Q]=b$.

Notice that if $n\ge 5$, then $b>2$. Indeed, by Corollary~\ref{cor:finite-g}(iii), every prime divisor of $g$ divides $n$ and for $p\mid n$,
\[
e_p=v_p(g)\le \frac{n}{2}v_p(n).
\]
If $n\ge 5$, Lemma~\ref{lem:vp-ineq} gives $v_p(n)<\frac{n-1}{2}$, hence
\[
e_p<\frac{n}{2}\cdot\frac{n-1}{2}=\frac{N}{2}.
\]
Therefore $d\le \max_p e_p<\frac{N}{2}$, whence $d_0=\gcd(N,d)\le d<\frac{N}{2}$ and thus
$b=N/d_0>2$.

Let $n\ge 5$. Because $g>0$, the real $b$th root of $h^a$ gives an embedding $\Q((h^a)^{1/b})\hookrightarrow\R$. If $\Q((h^a)^{1/b})/\Q$ were Galois and $b>2$, then it would be normal and hence contain all conjugates
$(h^a)^{1/b}\zeta_b^k$; in particular it would contain $\zeta_b$.
But $\zeta_b\notin\R$ for $b>2$, contradicting $\Q((h^a)^{1/b})\subseteq\R$.

Let $n=4$. Here $N=6$ and $K=\Q(\zeta_{12})$ has degree $4$.
Corollary~\ref{cor:finite-g}(ii)--(iii) implies $g$ is a power of $2$ with $1\le g\le 16$, i.e.\ $g\in\{2,4,8,16\}$.
If $g\in\{2,4,16\}$, then $[\Q(g^{1/6}):\Q]\in\{6,3\}$, so $\Q(g^{1/6})\not\subseteq K$.
If $g=8$, then $g^{1/6}=2^{1/2}$ and $\Q(g^{1/6})=\Q(\sqrt2)$.
But $\Q(\zeta_{12})=\Q(i,\sqrt3)$ has quadratic subfields $\Q(i)$, $\Q(\sqrt3)$, and $\Q(\sqrt{-3})$,
so $\Q(\sqrt2)\not\subseteq K$.

Thus in all cases $g^{1/N}\notin K$, contradicting $L=K$.
\end{proof}

%==========================================================
\section{Density consequences in the pure family}
\label{sec:density-pure-family}
%==========================================================

%==========================================================
\subsection{Density results}
\label{sec:many-prime-sieve}
%==========================================================

Fix an integer $n\ge 4$ and set $N=\frac{n(n-1)}2$.
For square-free $m$ with $X^n-m$ irreducible over $\Q$, write $K_m=\Q(\alpha)$ with $\alpha^n=m$,
write $\Ocal_m=\Ocal_{K_m}$, and let
\[
g(m):=[\Ocal_m:\Z[\alpha]].
\]
By Corollary~\ref{cor:finite-g}, the set of possible values of $g(m)$ is finite; denote it by $G(n)$.

For each $g\in G(n)$ with $g\ge 2$, define $P_g$ as in Definition~\ref{def:Pg-general}.
By Lemma~\ref{lem:kummer-nontrivial} we have $K=\Q(\zeta_{2N})\subsetneq \Q(\zeta_{2N},g^{1/N})$,
so Proposition~\ref{prop:Pg-positive-density-general} applies and gives $\sum_{q\in P_g}1/q=\infty$. We show that primes in $P_g$ force an ABS local obstruction at $q\mid m$.

\begin{corollary}
\label{cor:Pg-obstruct}
Let $m$ be square-free with $X^n-m$ irreducible and suppose $g(m)=g\ge 2$.
If $q\in P_g$ divides $m$, then for every orientation $\omega$ of $\Ocal_m$ and every $\varepsilon\in\{\pm 1\}$,
the local equation
\[
f_{m,\omega}(x)=\varepsilon
\]
has no solution in $\Ocal_m\otimes_{\Z}\Z_q$.
Equivalently, $K_m$ has an Alp\"oge--Bhargava--Shnidman fixed-sign local obstruction at $q$.
\end{corollary}

\begin{proof}
Fix an orientation $\omega$ of $\Ocal_m$ and $\varepsilon\in\{\pm1\}$.
Assume for contradiction that $f_{m,\omega}(x)=\varepsilon$ has a solution in $\Ocal_m\otimes\Z_q$.

Since $q\mid m$ and $m$ is square-free, the polynomial $X^n-m$ is Eisenstein at $q$, hence
$\Ocal_m\otimes\Z_q$ is a DVR of residue degree $1$ and is generated over $\Z_q$ by the uniformizer $\alpha$.
Moreover, $q\in P_g$ implies $q\equiv 1\pmod{2N}$, so in particular $q\nmid N$.
Applying the local Kummer rigidity (Theorem~\ref{thm:localcoset}) to $A=\Ocal_m\otimes\Z_q$ shows that
solvability for the unit right-hand side $\varepsilon$ forces
\[
g\in \pm(\Z_q^\times)^N.
\]
Reducing modulo $q$ gives $g\in \pm(\F_q^\times)^N$.
But since $q\equiv 1\pmod{2N}$, Corollary~\ref{cor:sign-killed} gives $\pm(\F_q^\times)^N=(\F_q^\times)^N$,
so $g\in(\F_q^\times)^N$, contradicting the defining condition $q\in P_g$.
\end{proof}

We show that the density of fixed-sign local solvability outside the $\alpha$--monogenic locus is zero.
\begin{theorem}
\label{thm:density-zero-ABS}
Let $n\ge 4$.  Let $S'$ be the set of square-free integers $m$ such that $X^n-m$ is irreducible and
$K_m$ is fixed-sign locally solvable for monogenicity in the sense of Definition~\ref{def:ABS},
but with $g(m)\ge 2$ (equivalently $\Ocal_m\neq \Z[\alpha]$).
Then $\delta(S')=0$.
\end{theorem}

\begin{proof}
Write $S'=\bigcup_{g\in G(n),\, g\ge 2} S'_g$, where $S'_g$ consists of those $m\in S'$ with $g(m)=g$.
It suffices to show $\delta(S'_g)=0$ for each fixed $g\ge 2$.

Fix $g\ge 2$ and let $m\in S'_g$.
By definition of fixed-sign local solvability, there exist an orientation $\omega$ and $\varepsilon\in\{\pm1\}$ such that
$f_{m,\omega}(x)=\varepsilon$ is solvable over $\Z_\ell$ for every prime $\ell$, hence in particular over $\Z_q$
for every prime $q\mid m$.

If $q\in P_g$ divides $m$, then Corollary~\ref{cor:Pg-obstruct} contradicts local solvability at $q$.
Hence every $m\in S'_g$ is divisible by no prime in $P_g$, i.e.\ every $m\in S'_g$ is $P_g$-free.

By Proposition~\ref{prop:Pg-positive-density-general}, $\sum_{q\in P_g}1/q=\infty$.
Therefore Lemma~\ref{lem:avoid-divergent-P} implies the set of $P_g$-free integers has density $0$, and thus
$\delta(S'_g)=0$. Since the union over $g$ is finite, $\delta(S')=0$.
\end{proof}

Note that monogenicity implies no fixed-sign local obstruction as follows.
\begin{lemma}
\label{lem:monogenic-implies-no-local-obstruction}
Let $K$ be a number field of degree $n\ge 2$ with ring of integers $\Ocal_K$.
If $\Ocal_K$ is monogenic, i.e.\ $\Ocal_K=\Z[\theta]$ for some $\theta\in\Ocal_K$, then $K$ is fixed-sign locally solvable
for monogenicity in the sense of Definition~\ref{def:ABS}.
\end{lemma}

\begin{proof}
Assume $\Ocal_K=\Z[\theta]$ and set $\omega_\theta:=1\wedge\theta\wedge\cdots\wedge\theta^{n-1}$.
Then $\omega_\theta$ is an orientation of $\Ocal_K$ and with $\omega=\omega_\theta$ the index form satisfies $f_{K,\omega}(\theta)=1$.
For every prime $\ell$, viewing $\theta$ in $\Ocal_K\otimes\Z_\ell$ gives a $\Z_\ell$-solution to $f_{K,\omega}(x)=1$.
\end{proof}

We show the Density equality between monogenicity vs.\ $\alpha$--monogenicity among square-free parameters.
\begin{theorem}
\label{thm:density-equality}
Fix $n\ge 4$. Among square-free $m$ with $X^n-m$ irreducible, the set of parameters for which $\Ocal_m$ is monogenic
but $\Ocal_m\neq\Z[\alpha]$ has two-sided natural density $0$.
\end{theorem}

\begin{proof}
If $\Ocal_m$ is monogenic but $\Ocal_m\neq\Z[\alpha]$, then $g(m)\ge 2$.
By Lemma~\ref{lem:monogenic-implies-no-local-obstruction}, $K_m$ is fixed-sign locally solvable.
Hence such $m$ lie in the set $S'$ of Theorem~\ref{thm:density-zero-ABS}, which has density $0$.
\end{proof}

Consequently, let
\[
\mathcal{M}_{\mathrm{sf}}:=\{m\in\Z:\ m\ \text{square-free and }X^n-m\ \text{irreducible over }\Q\}.
\]
Let
\[
\mathcal{M}_{\mathrm{sf}}^{\mathrm{mono}}:=\{m\in\mathcal{M}_{\mathrm{sf}}:\ \Ocal_{K_m}\ \text{is monogenic}\}.
\]
By \cite{NDN25}, the density of $\alpha$--monogenicity in $\mathcal{M}_{\mathrm{sf}}$ is
\[
\delta\!\left(\{m\in\mathcal{M}_{\mathrm{sf}}:\ \Ocal_{K_m}=\Z[\alpha]\}\right)
=\frac{6}{\pi^2}\prod_{p\mid n}\frac{p}{p+1}.
\]
Then
\[
\delta\!\left(\mathcal{M}_{\mathrm{sf}}^{\mathrm{mono}}\right)
=\frac{6}{\pi^2}\prod_{p\mid n}\frac{p}{p+1}.
\]

%==========================================================
\subsection{Low-degree examples}
\label{subsec:explicit-examples}
%==========================================================

Recall that, from Dedekind's index theorem, we obtained the following criterion for being $\alpha$-monogenity in the pure family.

\begin{theorem}[{\cite[Theorem~2.4]{NDN25}}]
\label{thm:alpha-criterion}
Let $n\ge 2$ and let $m\in\Z\setminus\{0\}$ be such that $X^n-m$ is irreducible over $\Q$.
Let $\alpha^n=m$ and $K_m=\Q(\alpha)$.
Then
\[
\Ocal_m=\Z[\alpha]
\quad\Longleftrightarrow\quad
\text{$m$ is square-free and $\vvp_p(m^p-m)=1$ for every prime $p\mid n$.}
\]
\end{theorem}

\noindent\textbf{Maximality of the power order in degrees $4$.}
\begin{equation}\label{eq:alpha-mono-small-n}
\Ocal_{\Q(\sqrt[4]{m})}=\Z[\sqrt[4]{m}]
\quad\Longleftrightarrow\quad m\not\equiv 1\pmod 4
\end{equation}
Equivalently, $m^p\not\equiv m\pmod{p^2}$ for every prime $p\mid n$.

%----------------------------------------------------------
\begin{example}[Quartic case:]
\label{ex:quartic-73-13}
Let $n=4$ and $\alpha^4=m$.

\smallskip\noindent
(1) (\emph{$\alpha$--monogenic.})
If $m=2$, then $2\not\equiv 1\pmod 4$, so $\Ocal_{\Q(\alpha)}=\Z[\alpha]$ by \eqref{eq:alpha-mono-small-n}.

\smallskip\noindent
(2) (\emph{Monogenic but not $\alpha$--monogenic.})
If $m=73$, then $73\equiv 1\pmod 4$, so $\Z[\alpha]\neq \Ocal_{\Q(\alpha)}$ by \eqref{eq:alpha-mono-small-n}.
Nevertheless, explicit solutions of the quartic index form equation are known in the residue class
$m\equiv 9\pmod{16}$: in particular, $K=\Q(\sqrt[4]{73})$ (and also $K=\Q(\sqrt[4]{89})$) admits an element of
index $1$ (and hence is monogenic), while evidence suggests such examples are extremely sparse within that class.
(See, e.g., \cite[Rem.~5]{GaalRemete2017}.)

\smallskip\noindent
(3) (\emph{A concrete Eisenstein-prime obstruction to ABS fixed-sign local solubility.})
Let $m=13$ and $\alpha^4=13$.  Since $13\equiv 5\pmod 8$, an explicit integral basis for
$\Ocal_{\Q(\alpha)}$ is (cf.\ \cite[\S4, Case $m=5+8k$]{GaalRemete2017})
\[
\Bigl\{\,1,\ \alpha,\ \frac{1+\alpha^2}{2},\ \frac{\alpha+\alpha^3}{2}\Bigr\}.
\]
Relative to the power basis $(1,\alpha,\alpha^2,\alpha^3)$, the change-of-basis matrix is upper triangular
with diagonal $(1,1,\tfrac12,\tfrac12)$, hence
\[
g(13):=[\Ocal_{\Q(\alpha)}:\Z[\alpha]]=4.
\]

Now $N=\frac{4\cdot3}{2}=6$, and $q=13$ is an Eisenstein prime for $X^4-13$.
Moreover $13\equiv 1\pmod{2N}=1\pmod{12}$ and $13\nmid 2Ng(13)$.
Since $\F_{13}^\times$ is cyclic of order $12$, its subgroup of $6$th powers has order
$12/\gcd(12,6)=2$, namely
\[
(\F_{13}^\times)^6=\{\pm1\}.
\]
In particular $g(13)=4\notin(\F_{13}^\times)^6$, so $13\in P_{g(13)}$ in the sense of
Definition~\ref{def:Pg-general}.
By the local coset $\Rightarrow$ residue-field power mechanism at Eisenstein primes
(cf.\ Corollary~\ref{cor:Pg-obstruct}),
it follows that for \emph{every} orientation $\omega$ of $\Ocal_{\Q(\alpha)}$ and every $\varepsilon\in\{\pm1\}$
the local equation
\[
f_{m,\omega}(x)=\varepsilon
\]
has no solution over $\Z_{13}$.  Equivalently, $K_{13}$ has an ABS fixed-sign local obstruction at $13$.
\end{example}

%==========================================================
\section{Quantitative sieve closure for a fixed index}
\label{sec:quantitative-sieve}
%==========================================================

Fix $n\ge 4$ and put $N=\frac{n(n-1)}2$.
Let $g\ge 2$ be an integer such that, for $K:=\Q(\zeta_{2N})$,
\[
K\subsetneq L:=\Q(\zeta_{2N},g^{1/N}).
\]
Define
\[
P_g:=\Bigl\{\,q\ \text{prime}:\ q\nmid 2Ng,\ q\equiv 1\!\!\pmod{2N},\ g\notin(\F_q^\times)^N\,\Bigr\}.
\]
For $X\ge 2$, define $\mathcal A_g(X)$ to be the set of integers $m$ such that:
\begin{itemize}[leftmargin=*]
\item $|m|\le X$;
\item $m$ is square-free;
\item $X^n-m$ is irreducible over $\Q$ and $g(m)=g$;
\item $K_m$ is fixed-sign locally solvable for monogenicity (Definition~\ref{def:ABS}).
\end{itemize}

We can show log-power upper bound for fixed $g$.
\begin{proposition}
\label{prop:logpower-fixed-g}
With notation as above, set
\[
\delta_g:=\frac{1}{[K:\Q]}-\frac{1}{[L:\Q]}>0.
\]
Then
\[
\#\mathcal A_g(X)\ \ll_{n,g}\ \frac{X}{(\log X)^{\delta_g}}
\qquad (X\to\infty).
\]
Moreover, since $[K:\Q]=\varphi(2N)$ and $[L:K]\ge 2$, one has
\[
\delta_g=\frac{[L:K]-1}{[L:K]\,[K:\Q]}\ \ge\ \frac{1}{2\varphi(2N)},
\]
so a uniform exponent $\frac{1}{2\varphi(2N)}$ works for all such $g$.
\end{proposition}

\begin{proof}

Call an integer $m$ \emph{$P_g$-free} if $q\nmid m$ for every $q\in P_g$.
Let $m\in\mathcal A_g(X)$. If $q\in P_g$ divides $m$, then $v_q(m)=1$ and hence $X^n-m$ is Eisenstein at $q$.
Also $q\equiv 1\pmod{2N}$ implies $q>2N\ge N$, so $q\nmid N$.
Therefore Corollary~\ref{cor:Pg-obstruct} applies and gives an ABS local obstruction at $q$,
contradicting $m\in\mathcal A_g(X)$. Hence every $m\in\mathcal A_g(X)$ is $P_g$-free, so
\begin{equation}
\label{eq:AgX-le-Pfree}
\#\mathcal A_g(X)\ \le\ \#\{m\in\Z:\ |m|\le X,\ \text{$m$ is $P_g$-free}\}.
\end{equation}

Let $G:=\Gal(L/\Q)$ and let $H:=\Gal(L/K)$, so that $H\lhd G$ and $|G|=[L:\Q]$, $|H|=[L:K]$.
Set $C:=H\setminus\{1\}\subseteq G$, which is conjugacy-stable.
Chebotarev implies that the set of unramified primes $q$ with $\Frob_q(L/\Q)\subseteq C$ has density
\[
\frac{|C|}{|G|}=\frac{|H|-1}{|G|}=\frac{1}{[K:\Q]}-\frac{1}{[L:\Q]}=\delta_g.
\]
By Lemma~\ref{lem:kummer-splitting-general}, the set $P_g$ differs by finitely many primes from this Chebotarev set.

Write $\pi_{P_g}(y):=\#\{q\le y:\ q\in P_g\}$.
An effective Chebotarev prime number theorem (e.g.\ Lagarias--Odlyzko \cite{LO77}, see also Serre's exposition \cite[\S3]{SerreChebotarev})
gives an asymptotic of the form
\begin{equation}
\label{eq:chebPNT-Pg}
\pi_{P_g}(y)
=
\delta_g\,\Li(y)
\ +\ O_{n,g}\!\bigl(\Li(y^{\beta_0})\bigr)
\ +\ O_{n,g}\!\Bigl(y\,e^{-c\sqrt{\log y}}\Bigr)
\qquad (y\to\infty),
\end{equation}
where $c=c(n,g)>0$, and the term $\Li(y^{\beta_0})$ appears only if there is an exceptional real zero $\beta_0\in(0,1)$.

A partial-summation identity yields
\[
\sum_{\substack{q\le y\\ q\in P_g}}\frac{1}{q}
=
\frac{\pi_{P_g}(y)}{y}+\int_{2}^{y}\frac{\pi_{P_g}(t)}{t^2}\,dt.
\]
Inserting \eqref{eq:chebPNT-Pg} shows that the main term contributes
\[
\frac{\delta_g\Li(y)}{y}+\delta_g\int_{2}^{y}\frac{\Li(t)}{t^2}\,dt
=
\delta_g\log\log y+O_{n,g}(1),
\]
while the error terms contribute $O_{n,g}(1)$ because
$\int_2^\infty e^{-c\sqrt{\log t}}t^{-1}\,dt<\infty$ and
$\int_2^\infty \Li(t^{\beta_0})t^{-2}\,dt<\infty$ for $\beta_0<1$.
Hence
\begin{equation}
\label{eq:recip-sum-Pg}
\sum_{\substack{q\le y\\ q\in P_g}}\frac{1}{q}
=
\delta_g\log\log y+O_{n,g}(1)
\qquad (y\to\infty).
\end{equation}
Define the truncated Mertens product
\[
V_g(y):=\prod_{\substack{q\le y\\ q\in P_g}}\Bigl(1-\frac{1}{q}\Bigr).
\]
Using $\log(1-1/q)=-1/q+O(1/q^2)$ and $\sum_q 1/q^2<\infty$, \eqref{eq:recip-sum-Pg} implies
\begin{equation}
\label{eq:Vg-logpower}
V_g(y)\ \ll_{n,g}\ (\log y)^{-\delta_g}
\qquad (y\to\infty).
\end{equation}

For $X\ge 2$ and $2\le y\le X$, set
\[
S(X,y):=\#\Bigl\{\,1\le m\le X:\ q\nmid m\ \text{for all }q\in P_g\text{ with }q\le y\,\Bigr\}.
\]
Since $\#\{1\le m\le X:\ d\mid m\}=X/d+O(1)$ uniformly in squarefree $d$, the fundamental lemma of the
dimension-$1$ sieve (e.g.\ the linear sieve; see \cite[\S6.4]{IK04} or \cite[Ch.~5]{HR74}) implies that for any fixed $u>1$
and any $y\le X^{1/u}$ one has
\[
S(X,y)\ \ll_{u}\ X\,V_g(y).
\]
Choose $u=2$ and $y=X^{1/2}$. Combining with \eqref{eq:Vg-logpower} and $\log y\asymp\log X$ gives
\[
S(X,X^{1/2})\ \ll_{n,g}\ \frac{X}{(\log X)^{\delta_g}}.
\]
Every $P_g$-free integer $1\le m\le X$ is counted by $S(X,X^{1/2})$, hence
\[
\#\{1\le m\le X:\ \text{$m$ is $P_g$-free}\}\ \ll_{n,g}\ \frac{X}{(\log X)^{\delta_g}}.
\]
The same bound holds for negative $m$ by symmetry, so
\[
\#\{m\in\Z:\ |m|\le X,\ \text{$m$ is $P_g$-free}\}
\ \ll_{n,g}\ \frac{X}{(\log X)^{\delta_g}}.
\]
Combining with \eqref{eq:AgX-le-Pfree} proves the proposition.
\end{proof}

\begin{remark}
If one wishes to avoid invoking an effective Chebotarev theorem, it suffices for the sieve bound
to have a lower bound of the shape
\[
\sum_{\substack{q\le y\\ q\in P_g}}\frac1q \ge \frac{\delta_g}{2}\log\log y
\qquad (y\ge y_0(n,g)),
\]
which yields $V_g(y)\ll_{n,g}(\log y)^{-\delta_g/2}$ and hence
$\#\mathcal A_g(X)\ll_{n,g}X/(\log X)^{\delta_g/2}$.
\end{remark}

As shown earlier in the paper, $P_g$ differs by finitely many primes from the Chebotarev set
$\{q:\ \Frob_q(L/\Q)\subseteq H\setminus\{1\}\}$ where $H=\Gal(L/K)$.
%==========================================================
% (1) Full Mertens asymptotic for P_g
%==========================================================
We can state the Mertens theorem for the Chebotarev--Kummer obstruction set $P_g$ as follows.
\begin{proposition}
\label{prop:mertens-Pg-constant}
Assume $L\ne K$.
Then there exist constants $B_g\in\R$ and $C_g>0$ such that, as $x\to\infty$,
\begin{align}
\sum_{\substack{q\le x\\ q\in P_g}}\frac1q
&=\delta_g\log\log x+B_g+o(1),
\label{eq:Pg-mertens-sum}
\\
\prod_{\substack{q\le x\\ q\in P_g}}\Bigl(1-\frac1q\Bigr)
&\sim \frac{C_g}{(\log x)^{\delta_g}}.
\label{eq:Pg-mertens-prod}
\end{align}
In particular, $\sum_{q\in P_g}\frac1q=\infty$.
\end{proposition}

\begin{proof}
Let $G:=\Gal(L/\Q)$ and $H:=\Gal(L/K)$, so that $H\lhd G$. For every unramified rational prime $q$,
write $\Frob_q(L/\Q)\subseteq G$ for the Frobenius conjugacy class.

By the Kummer splitting criterion proved earlier in the paper, for primes $q\nmid 2Ng$ the conditions
\[
q\equiv 1\pmod{2N}
\quad\text{and}\quad
g\notin(\F_q^\times)^N
\]
are equivalent to: $q$ splits completely in $K$ but does not split completely in $L$.
Equivalently, $\Frob_q(L/\Q)\subseteq H\setminus\{1\}$.
Thus $P_g$ differs by at most finitely many primes (those dividing $2Ng\disc(L)$) from the Chebotarev set
\[
\mathcal{P}_{H\setminus\{1\}}:=\{q\ \text{unramified in }L:\ \Frob_q(L/\Q)\subseteq H\setminus\{1\}\}.
\]

Write $H\setminus\{1\}$ as a disjoint union of conjugacy classes $C_1,\dots,C_r$ in $G$.
For each class $C_i$, Arango--Pi\~neros--Keliher--Keyes prove a Mertens product theorem for the Chebotarev
set $\{q:\ \Frob_q(L/\Q)=C_i\}$, with a nonzero constant and an asymptotic of the form
\[
\prod_{\substack{q\le x\\ \Frob_q(L/\Q)=C_i}}\Bigl(1-\frac1q\Bigr)
\sim \frac{c(C_i)}{(\log x)^{|C_i|/|G|}}
\qquad (x\to\infty),
\]
see \cite[Theorem~A]{AKK22}.
Multiplying over $i=1,\dots,r$ (the sets are disjoint) yields
\[
\prod_{\substack{q\le x\\ \Frob_q(L/\Q)\subseteq H\setminus\{1\}}}\Bigl(1-\frac1q\Bigr)
\sim \frac{C_g'}{(\log x)^{\sum_i |C_i|/|G|}}
=\frac{C_g'}{(\log x)^{(|H|-1)/|G|}}
\qquad (x\to\infty),
\]
for some constant $C_g'>0$.
Since $|G|=[L:\Q]$ and $|H|=[L:K]=[L:\Q]/[K:\Q]$, we have
\[
\frac{|H|-1}{|G|}
=\frac1{[K:\Q]}-\frac1{[L:\Q]}
=\delta_g.
\]
Inserting the finitely many excluded primes changes the product by a nonzero factor, hence
\eqref{eq:Pg-mertens-prod} holds for some $C_g>0$.

Taking logarithms in \eqref{eq:Pg-mertens-prod} gives
\[
\sum_{\substack{q\le x\\ q\in P_g}}\log\Bigl(1-\frac1q\Bigr)
= -\delta_g\log\log x + \log C_g + o(1).
\]
Using $\log(1-u)=-u+O(u^2)$ and $\sum_q q^{-2}<\infty$, we have
\[
\sum_{\substack{q\le x\\ q\in P_g}}\log\Bigl(1-\frac1q\Bigr)
=
-\sum_{\substack{q\le x\\ q\in P_g}}\frac1q + O(1).
\]
Comparing the last two displays yields \eqref{eq:Pg-mertens-sum} for some constant $B_g\in\R$.
In particular, $\sum_{q\in P_g}1/q=\infty$.
\end{proof}

%==========================================================
% (2) Optional: asymptotic for P_g-free integers (Wirsing/Delange)
%==========================================================
We get the asymptotic count of $P_g$-free integers.
\begin{corollary}
\label{cor:Pg-free-asymptotic-wirsing}
Assume the hypotheses of Proposition~\ref{prop:mertens-Pg-constant}. Define
\[
\mathbf{1}_{P_g\text{-free}}(m):=
\begin{cases}
1,&\text{if $q\nmid m$ for all $q\in P_g$},\\
0,&\text{otherwise.}
\end{cases}
\]
Then there exists a constant $\kappa_g>0$ such that
\[
\sum_{1\le m\le X}\mathbf{1}_{P_g\text{-free}}(m)
\sim
\kappa_g\,\frac{X}{(\log X)^{\delta_g}}
\qquad (X\to\infty).
\]
\end{corollary}

\begin{proof}
Set $f(m):=\mathbf{1}_{P_g\text{-free}}(m)$. Then $f$ is multiplicative and satisfies
\[
f(p^k)=
\begin{cases}
0,&p\in P_g,\\
1,&p\notin P_g,
\end{cases}
\qquad(k\ge 1).
\]
In particular, $0\le f(n)\le 1$.

We apply Wirsing’s mean value theorem in the form of Indlekofer \cite[Theorem~(Satz)~1.1]{Ind80}.
It suffices to verify the standard hypotheses:
\begin{enumerate}[label=\textup{(\roman*)},leftmargin=*]
\item $f(p)=O(1)$ for primes $p$ (trivial) and $\sum_{p}\sum_{k\ge 2} f(p^k)/p^k<\infty$ (true since $f(p^k)\le 1$);
\item there exists $\alpha>0$ such that $\sum_{p\le x} f(p)\frac{\log p}{p}\sim \alpha\log x$.
\end{enumerate}
For (ii), note
\[
\sum_{p\le x} f(p)\frac{\log p}{p}
=
\sum_{p\le x}\frac{\log p}{p}
-\sum_{\substack{p\le x\\ p\in P_g}}\frac{\log p}{p}.
\]
The classical estimate $\sum_{p\le x}\frac{\log p}{p}=\log x+O(1)$ is standard.
Moreover, Proposition~\ref{prop:mertens-Pg-constant} gives
$\sum_{p\le x,\,p\in P_g}\frac1p=\delta_g\log\log x+O(1)$, and partial summation yields
\[
\sum_{\substack{p\le x\\ p\in P_g}}\frac{\log p}{p}
=
\delta_g\log x+O(1).
\]
Hence
\[
\sum_{p\le x} f(p)\frac{\log p}{p}
=
(1-\delta_g)\log x+O(1),
\]
so Indlekofer applies with $\alpha=1-\delta_g$.
The conclusion of \cite[Theorem~(Satz)~1.1]{Ind80} gives the claimed asymptotic with exponent $\delta_g=1-\alpha$.
\end{proof}

%----------------------------------------------------------
\section{On scaled Eisenstein families}
\label{sec:examples-counterexamples}
%----------------------------------------------------------

This section records two complementary phenomena.
First, we package a broad class of \emph{scaled Eisenstein} families in which the Eisenstein--prime
obstruction sieve applies verbatim once one has bounded index values and nontrivial Kummer data.
Second, we record thin-parameter constructions which show that (i) the Kummer nontriviality hypothesis
in the sieve is genuinely necessary, and (ii) in general one-parameter families, the \textsc{ABS}
fixed--sign local condition does not force the \emph{distinguished} generator to have index~$1$.

Throughout, $n\ge 2$ is an integer and
\[
N \ :=\ \frac{n(n-1)}2.
\]

%----------------------------------------------------------
\subsection{Scaled Eisenstein polynomial families}
\label{subsec:scaled-family}
%----------------------------------------------------------

Fix a polynomial
\[
h(X)=c_{n-1}X^{n-1}+c_{n-2}X^{n-2}+\cdots+c_1X+c_0\in\Z[X],
\qquad c_0\neq 0.
\]
For each integer parameter $t\in\Z$, consider the monic polynomial
\[
f_t(X)\ :=\ X^n+t\,h(X)\ \in\ \Z[X],
\]
and let $\theta_t$ be a root.  Set
\[
K_t:=\Q(\theta_t),\qquad \Ocal_t:=\Ocal_{K_t},\qquad g(t):=[\Ocal_t:\Z[\theta_t]]\in\Z_{\ge 1}.
\]
We say that $K_t$ is \emph{$\theta_t$--monogenic} if $g(t)=1$.

We work on the square-free parameter set
\[
\mathcal{T}_{h,\mathrm{sf}}
:=\Bigl\{\,t\in\Z:\ |t|>1,\ t\ \text{square-free, and }\gcd(t,c_0)=1\,\Bigr\}.
\]

We obtain the portability of the local coset constraint in the scaled Eisenstein family.
\begin{proposition}
\label{prop:scaled-portability-merged}
Let $t\in\mathcal{T}_{h,\mathrm{sf}}$ and let $q\mid t$ be a prime.  Then:
\begin{enumerate}[label=\textup{(\roman*)}]
\item \textup{(Eisenstein at $q$)} The polynomial $f_t(X)$ is Eisenstein at $q$.
In particular $f_t$ is irreducible over $\Q$ and $[K_t:\Q]=n$.

\item \textup{(Local identification and trivial $q$--part of the index)} There is a unique prime of $\Ocal_t$ above $q$,
the extension $K_t\otimes_\Q\Q_q/\Q_q$ is totally ramified of degree $n$ and residue degree $1$, and one has
a canonical identification of complete local rings
\[
\Ocal_t\otimes_\Z \Z_q \ \cong\ \Z_q[\theta_t].
\]
In particular, $\Ocal_t\otimes_\Z\Z_q$ is a DVR of rank $n$ over $\Z_q$ generated by the uniformizer $\theta_t$,
so the local index at $q$ equals $1$ and $q\nmid g(t)$ (equivalently, $g(t)\in\Z_q^\times$).

\item \textup{(Coset constraint for unit values of the index form)} Fix an orientation $\omega$ of $\Ocal_t$
and write $f_{t,\omega}$ for the associated index form.
Let $\omega_q$ be the induced orientation on $\Ocal_t\otimes\Z_q$.
We continue to write $f_{t,\omega}$ for its base change to $\Z_q$.
Assume moreover that $q\nmid N$.
Then for every $\beta\in \Ocal_t\otimes\Z_q$ with $\Z_q[\beta]=\Ocal_t\otimes\Z_q$ one has
\[
f_{t,\omega}(\beta)\in \pm\, g(t)\cdot (\Z_q^\times)^N.
\]
Consequently, if for some $\varepsilon\in\{\pm 1\}$ the local equation $f_{t,\omega}(x)=\varepsilon$
is solvable over $\Z_q$, then
\[
g(t)\in \pm(\Z_q^\times)^N
\qquad\text{and hence}\qquad
\overline{g(t)}\in \pm(\F_q^\times)^N.
\]

\item \textup{(Killing the sign when $q\equiv 1\!\!\pmod{2N}$)} If in addition $q\equiv 1\pmod{2N}$,
then $-1\in(\F_q^\times)^N$ (and also $-1\in(\Z_q^\times)^N$), so the sign is irrelevant and local solvability forces
\[
\overline{g(t)}\in(\F_q^\times)^N.
\]
\end{enumerate}
\end{proposition}

\begin{proof}
(i) Since $q\mid t$, every non-leading coefficient of $f_t$ is divisible by $q$.
The constant term is $t\,c_0$, and because $t$ is square-free and $\gcd(t,c_0)=1$ we have
$v_q(t\,c_0)=1$. Hence $f_t$ is Eisenstein at $q$.

(ii) Eisenstein at $q$ implies $\Q_q(\theta_t)/\Q_q$ is totally ramified of degree $n$ with residue degree $1$,
and $\Z_q[\theta_t]$ is integrally closed in $\Q_q(\theta_t)$; hence $\Z_q[\theta_t]$ is the ring of integers
of $\Q_q(\theta_t)$.
Since there is a unique prime of $\Ocal_t$ above $q$, the factor $\Ocal_t\otimes\Z_q$ identifies canonically with this
local ring, yielding $\Ocal_t\otimes\Z_q\simeq \Z_q[\theta_t]$.
In particular the $q$-part of the global index is trivial, so $q\nmid g(t)$.

(iii) Let $\omega_q$ denote the image of $\omega$ in $\bigwedge_{\Z_q}^n(\Ocal_t\otimes\Z_q)$.
By the standard top-wedge change-of-lattice computation,
\[
f_{t,\omega_q}(\theta_t)=\pm[\Ocal_t:\Z[\theta_t]]=\pm g(t)\in\Z_q^\times.
\]
Apply the local coset constraint (Theorem~\ref{thm:localcoset}) to
$A=\Ocal_t\otimes\Z_q$ with uniformizer $\pi=\theta_t$ and orientation $\omega_q$.
Since $q\nmid N$, the unit value set on $\Z_q$--generators is exactly
$f_{t,\omega_q}(\theta_t)\cdot(\Z_q^\times)^N=\pm g(t)\cdot(\Z_q^\times)^N$.
If $f_{t,\omega}(x)=\varepsilon\in\{\pm1\}$ is solvable over $\Z_q$, then $\varepsilon\in\Z_q^\times$ and
Lemma~\ref{lem:unitwedge} implies $x$ is a local generator, hence $\varepsilon$ lies in this coset,
forcing $g(t)\in\pm(\Z_q^\times)^N$. Reducing modulo $q$ gives the residue-field condition.

(iv) If $q\equiv 1\pmod{2N}$ then $q>2N$, so $q\nmid N$ and $2N\mid (q-1)$.
Since $\F_q^\times$ is cyclic of order $q-1$, it follows that $-1\in (\F_q^\times)^N$; this lifts to
$-1\in(\Z_q^\times)^N$ because $q\nmid N$.
\end{proof}

\begin{corollary}[A one-prime \textsc{ABS} obstruction certificate in the scaled Eisenstein family]
\label{cor:scaled-one-prime-certificate-merged}
Let $t\in\mathcal{T}_{h,\mathrm{sf}}$ and suppose $g(t)=g\ge 2$.
Let $q\mid t$ be a prime with $q\equiv 1\pmod{2N}$.
If $g\notin(\F_q^\times)^N$, then $K_t$ has an \textsc{ABS} fixed-sign local obstruction at $q$:
for every orientation $\omega$ of $\Ocal_t$ and every $\varepsilon\in\{\pm1\}$,
the local equation $f_{t,\omega}(x)=\varepsilon$ has no solution over $\Z_q$.
\end{corollary}

\begin{proof}
If $f_{t,\omega}(x)=\varepsilon$ were solvable over $\Z_q$ for some $\omega$ and $\varepsilon\in\{\pm1\}$,
then Proposition~\ref{prop:scaled-portability-merged}(iv) would force $g\in(\F_q^\times)^N$,
contradicting the hypothesis.
\end{proof}

From here on assume $n\ge 4$ (so $N\ge 6$). We get the density zero for \textsc{ABS}-unobstructed non-$\theta_t$--monogenic parameters in the scaled Eisenstein family.

\begin{theorem}
\label{thm:scaled-density-zero-merged}
Fix $n\ge 4$ and $h\in\Z[X]$ as above, and keep $N=\frac{n(n-1)}2$.
Assume the following hypotheses on the square-free parameter set $\mathcal{T}_{h,\mathrm{sf}}$:
\begin{enumerate}[label=\textup{(\alph*)}]
\item \textup{(Finite index values)} There exists a finite set $G\subset\Z_{\ge1}$ such that
$g(t)\in G$ for all $t\in\mathcal{T}_{h,\mathrm{sf}}$.
\item \textup{(Kummer nontriviality)} For every $g\in G$ with $g\ge 2$,
\[
\Q(\zeta_{2N},g^{1/N})\neq \Q(\zeta_{2N}).
\]
\end{enumerate}
Let $\mathcal{S}'_{h}\subseteq\mathcal{T}_{h,\mathrm{sf}}$ be the set of $t$ such that:
\begin{enumerate}[label=\textup{(\roman*)}]
\item $g(t)\ge 2$, and
\item $K_t$ has \emph{no} local obstruction to monogenicity in the \textsc{ABS} fixed-sign sense.
\end{enumerate}
Then $\mathcal{S}'_{h}$ has two-sided natural density $0$ in $\Z$ (hence also relative density $0$
in $\mathcal{T}_{h,\mathrm{sf}}$).
\end{theorem}

\begin{proof}
Fix $g\in G$ with $g\ge 2$ and let $P_g$ be the Chebotarev--Kummer obstruction set
from Definition~\ref{def:Pg-general}.  By \textup{(b)} and Proposition~\ref{prop:Pg-positive-density-general},
$P_g$ has positive density among primes, hence $\sum_{q\in P_g}\frac1q=\infty$.

If $t\in\mathcal{S}'_h$ with $g(t)=g$ and some $q\in P_g$ divides $t$, then
Corollary~\ref{cor:scaled-one-prime-certificate-merged} yields an \textsc{ABS} local obstruction at $q$,
contradiction.  Hence such $t$ avoid all primes in $P_g$, and Lemma~\ref{lem:avoid-divergent-P}
forces density zero.  Taking the finite union over $g\in G$ completes the proof.
\end{proof}

\begin{corollary}
\label{cor:scaled-density-equality-merged}
Assume the hypotheses of Theorem~\ref{thm:scaled-density-zero-merged}.
Let
\[
\mathcal{S}_{\mathrm{mono}}
:=\{t\in\mathcal{T}_{h,\mathrm{sf}}:\ \Ocal_t\ \text{is monogenic}\},
\qquad
\mathcal{S}_{\theta}
:=\{t\in\mathcal{T}_{h,\mathrm{sf}}:\ \Ocal_t=\Z[\theta_t]\}.
\]
Then $\mathcal{S}_{\theta}\subseteq \mathcal{S}_{\mathrm{mono}}$ and the difference
$\mathcal{S}_{\mathrm{mono}}\setminus \mathcal{S}_{\theta}$ has two-sided natural density $0$.
Equivalently, within $\mathcal{T}_{h,\mathrm{sf}}$, monogenicity and $\theta_t$--monogenicity have the same
natural density (whenever either density exists).
\end{corollary}

\begin{proof}
If $t\in\mathcal{S}_{\mathrm{mono}}\setminus \mathcal{S}_{\theta}$ then $g(t)\ge 2$.
Moreover, monogenicity implies no \textsc{ABS} fixed-sign local obstruction
(Lemma~\ref{lem:monogenic-implies-no-local-obstruction}).
Hence $\mathcal{S}_{\mathrm{mono}}\setminus \mathcal{S}_{\theta}\subseteq \mathcal{S}'_{h}$, which has density $0$
by Theorem~\ref{thm:scaled-density-zero-merged}.
\end{proof}

%----------------------------------------------------------
\subsection{A thin-family construction showing the necessity of Kummer nontriviality}
\label{subsec:thin-kummer-necessary}
%----------------------------------------------------------

Our density--zero results for \emph{square-free} parameter sets use the following sieve closure mechanism:
for each fixed nontrivial index value $g\ge 2$, one produces an obstruction set of primes $P_g$ with
\[
\sum_{q\in P_g}\frac1q=\infty,
\]
and then uses the multiplicative structure of typical square-free integers to force most parameters
with many prime factors to meet $P_g$.  Two qualitatively different ways this can fail are:
\begin{enumerate}[label=\textup{(\alph*)}]
\item \emph{thin parameter sets} (primes, bounded--$\Omega$ almost-primes, polynomial values), where avoidance of a
positive-density set of primes need not be rare \emph{relative} to the thin set; and
\item \emph{Kummer triviality} for a given fixed index value $g$, in which case the Chebotarev--Kummer obstruction
set $P_g$ is empty (away from finitely many primes).
\end{enumerate}
The next lemma and theorem provide a uniform example illustrating \textup{(b)}.

\subsubsection{A thin set of primes with positive relative density among primes}

Fix integers
\[
n\ge 4,\qquad N:=\frac{n(n-1)}2,
\]
and fix an integer $c\ge 2$.
Let $\mathcal P_n$ be the set of primes $q$ such that:
\begin{enumerate}[label=\textup{(\roman*)}]
\item $q\nmid cn$, and
\item for every prime $p\mid n$ one has
\begin{equation}\label{eq:wieferich-avoid}
q^{p-1}\not\equiv 1\pmod{p^2}.
\end{equation}
\end{enumerate}

\begin{lemma}[Density of $\mathcal P_n$ among the primes]
\label{lem:density-Pm}
The set $\mathcal P_n$ has natural density $0$ in $\Z$, but it has a positive natural density inside the primes; in fact,
\[
\lim_{X\to\infty}\frac{\#\{q\le X:\ q\in\mathcal P_n\}}{\#\{q\le X:\ q\ \text{prime}\}}
\;=\;
\prod_{p\mid n}\Bigl(1-\frac1p\Bigr)
\;>\;0.
\]
\end{lemma}

\begin{proof}
Fix a prime $p\mid n$.  For odd $p$, the group $(\Z/p^2\Z)^\times$ is cyclic of order $\varphi(p^2)=p(p-1)$,
hence it has a unique subgroup of order $p-1$.
This subgroup is exactly the set
\[
\mathcal B_p:=\{a\in (\Z/p^2\Z)^\times:\ a^{p-1}\equiv 1\pmod{p^2}\},
\]
so $\#\mathcal B_p=p-1$. Therefore the set of \emph{allowed} reduced residue classes modulo $p^2$ has size
\[
\varphi(p^2)-\#\mathcal B_p=p(p-1)-(p-1)=(p-1)^2,
\]
i.e.\ an allowed proportion $1-\frac1p$.

If $p=2$ divides $n$, then $(\Z/4\Z)^\times=\{\pm1\}$ and the condition $q^{1}\not\equiv 1\pmod4$
is equivalent to $q\equiv 3\pmod4$, again of proportion $1-\frac12$.

Let $M:=\prod_{p\mid n}p^2$. By the Chinese remainder theorem, the simultaneous conditions
$q^{p-1}\not\equiv 1\pmod{p^2}$ for all $p\mid n$ amount to restricting $q$ to a subset of the reduced residue classes
modulo $M$ of proportion $\prod_{p\mid n}(1-\frac1p)$.
By the prime number theorem in arithmetic progressions (equivalently, Chebotarev in $\Q(\zeta_M)$),
primes are equidistributed among reduced residue classes modulo $M$, so the same proportion holds as a natural density
\emph{inside the primes}.  Finally, excluding the finitely many primes dividing $cn$ does not change the limiting density.
\end{proof}

\subsubsection{A uniform fixed-index but \textsc{ABS}-unobstructed construction}

For $q\in\mathcal P_n$, define
\[
f_q(X):=X^n-c^n q\in\Z[X],
\]
let $\theta_q$ be a root, and put $K_q:=\Q(\theta_q)$.
Set also
\[
\alpha_q:=\theta_q/c,
\]
so $\alpha_q^n=q$ and $K_q=\Q(\alpha_q)$.

\begin{theorem}[Thin-family construction: fixed index but \textsc{ABS}-unobstructed]
\label{thm:thin-counterexample}
Fix $n\ge 4$ and $c\ge 2$, and let $\mathcal P_n$ be as above.
By Theorem~\ref{thm:alpha-criterion}, for $q\in\mathcal P_n$ and $\alpha_q^n=q$ one has $\Ocal_{K_q}=\Z[\alpha_q]$.
Then for every $q\in\mathcal P_n$ the following hold:
\begin{enumerate}[label=\textup{(\roman*)}]
\item $\Ocal_{K_q}=\Z[\alpha_q]$ (so $K_q$ is monogenic);
\item the distinguished generator $\theta_q$ never yields a power integral basis:
\[
[\Ocal_{K_q}:\Z[\theta_q]]=c^N>1,
\qquad
N=\frac{n(n-1)}2;
\]
\item $K_q$ is \textsc{ABS}-unobstructed (fixed sign): there exists an orientation $\omega_q$ and a sign
$\varepsilon=+1$ such that the index form equation $f_{K_q,\omega_q}(x)=\varepsilon$ has a solution in $\Z_\ell$
for every prime $\ell$ (including $\ell=\infty$).
\end{enumerate}
In particular, the family $\{K_q\}_{q\in\mathcal P_n}$ is thin as a subset of $\Z$ (since it is indexed by primes),
yet it has positive density among primes by Lemma~\ref{lem:density-Pm}, and it has $100\%$ relative density
of parameters which are \textsc{ABS}-unobstructed while satisfying $[\Ocal_{K_q}:\Z[\theta_q]]>1$.
\end{theorem}

\begin{proof}
(1) This is exactly the conclusion of Theorem~\ref{thm:alpha-criterion} applied to $\alpha_q^n=q$.

(2) Since $\theta_q=c\alpha_q$, we have inclusions of orders
\[
\Z[\theta_q]\subseteq \Z[\alpha_q]=\Ocal_{K_q}.
\]
As $\Z$-modules,
\[
\Z[\alpha_q]=\Z\langle 1,\alpha_q,\alpha_q^2,\dots,\alpha_q^{n-1}\rangle,
\qquad
\Z[\theta_q]=\Z\langle 1,c\alpha_q,c^2\alpha_q^2,\dots,c^{n-1}\alpha_q^{n-1}\rangle.
\]
With respect to these bases, the inclusion $\Z[\theta_q]\hookrightarrow \Z[\alpha_q]$ is given by the diagonal
change-of-basis matrix $\mathrm{diag}(1,c,c^2,\dots,c^{n-1})$, hence
\[
[\Z[\alpha_q]:\Z[\theta_q]]
=\prod_{j=0}^{n-1}c^j
=c^{\sum_{j=0}^{n-1}j}
=c^{n(n-1)/2}
=c^N.
\]
Using $\Ocal_{K_q}=\Z[\alpha_q]$ from (1), we obtain $[\Ocal_{K_q}:\Z[\theta_q]]=c^N>1$.

(3) Since $\Ocal_{K_q}=\Z[\alpha_q]$, choose the standard orientation
\[
\omega_q:=1\wedge \alpha_q\wedge \alpha_q^2\wedge\cdots\wedge \alpha_q^{n-1}\in \bigwedge^n \Ocal_{K_q}.
\]
For this orientation, the associated index form satisfies $f_{K_q,\omega_q}(\alpha_q)=+1$.
Therefore the same integral point witnessing monogenicity gives a solution to $f_{K_q,\omega_q}(x)=+1$
over $\Z_\ell$ for every prime $\ell$ (including $\ell=\infty$), i.e.\ there is no fixed-sign local obstruction in the
sense of \textsc{ABS}.
\end{proof}

\begin{remark}[Why the Chebotarev--Kummer obstruction vanishes here]
\label{rem:kummer-trivial}
In the Chebotarev--Kummer obstruction mechanism for a fixed index value $g$, the relevant Kummer extension over
$\Q(\zeta_{2N})$ becomes trivial if $g\in (\Q^\times)^N$.
In the above construction,
\[
g(\theta_q)=[\Ocal_{K_q}:\Z[\theta_q]]=c^N\in (\Q^\times)^N,
\]
so $\Q(\zeta_{2N},g(\theta_q)^{1/N})=\Q(\zeta_{2N})$ and the corresponding Chebotarev obstruction set is empty
away from finitely many primes dividing $c$.
This shows that the Kummer nontriviality hypothesis in Theorem~\ref{thm:scaled-density-zero-merged} is essential.
\end{remark}

%----------------------------------------------------------
\subsection{Positive-density \textsc{ABS}-unobstructed families and a fixed-index twist}
\label{subsec:positive-density-ABS-families}
%----------------------------------------------------------

We now record a (thin) one-parameter family with \emph{positive natural density} of monogenic specializations,
together with a scaling twist which forces the \emph{distinguished} generator to have a fixed nontrivial index.

\subsubsection{A trinomial family with positive density of monogenic specializations}

Consider the one--parameter family
\[
f_t(X)\ :=\ X^n+tX+t\ \in\ \Z[X],\qquad K_t\ :=\ \Q(\alpha_t),
\]
where $\alpha_t$ is any root of $f_t$.

\begin{lemma}[Discriminant of $X^n+AX+B$]
\label{lem:trinomial-disc}
For $n\ge 2$ and $A,B\in\Z$, one has
\[
\disc(X^n+AX+B)
=
(-1)^{n(n-1)/2}n^n B^{n-1}
+
(-1)^{(n-1)(n-2)/2}(n-1)^{n-1}A^n.
\]
In particular, with
\[
C_{0,n}:=(-1)^{n(n-1)/2}n^n,\qquad C_{1,n}:=(-1)^{(n-1)(n-2)/2}(n-1)^{n-1},
\qquad
L_n(t):=C_{0,n}+C_{1,n}t,
\]
we have
\begin{equation}
\label{eq:disc-ft}
\disc(f_t)=t^{n-1}L_n(t).
\end{equation}
\end{lemma}

\begin{proof}
This is the standard resultant computation:
$\disc(X^n+AX+B)=(-1)^{n(n-1)/2}\Res(X^n+AX+B,\,nX^{n-1}+A)$.
\end{proof}

Define the arithmetic parameter set
\[
\mathcal T_n\ :=\ \Bigl\{t\in\Z:\ |t|>1,\ \gcd\bigl(t,n(n-1)\bigr)=1,\ \text{and }t\,L_n(t)\text{ is squarefree}\Bigr\}.
\]

\begin{theorem}[Positive density of monogenic (hence \textsc{ABS}--unobstructed) specializations]
\label{thm:ft-monogenic-positive-density}
For every $n\ge 4$, the set $\mathcal T_n$ has positive natural density in $\Z$.
Moreover, for every $t\in\mathcal T_n$, the order $\Z[\alpha_t]$ is maximal:
\[
\Ocal_{K_t}=\Z[\alpha_t].
\]
In particular, $K_t$ is monogenic and hence \textsc{ABS}--unobstructed for all $t\in\mathcal T_n$.
\end{theorem}

\begin{proof}
Fix $t\in\mathcal T_n$ and put $I_t:=[\Ocal_{K_t}:\Z[\alpha_t]]$.

Let $q\mid t$ be prime. Since $t$ is squarefree and $\gcd(t,n(n-1))=1$, we have $v_q(t)=1$ and $q\nmid n(n-1)$.
The polynomial $f_t(X)=X^n+tX+t$ is Eisenstein at $q$, hence irreducible over $\Q$.

Fix a prime $q\mid t$ and set
\[
R_q:=\Z_q[\alpha_t]\;\cong\;\Z_q[X]/(f_t).
\]
Since $f_t$ is Eisenstein at $q$, it is irreducible over $\Q_q$, hence $R_q$ is a domain, finite over the DVR $\Z_q$,
and therefore $\dim(R_q)=1$. Reducing modulo $q$ gives
\[
R_q/qR_q \cong \F_q[X]/(X^n),
\]
so $R_q$ is local with maximal ideal $\mathfrak m_q=(q,\alpha_t)$.

In $R_q$ we have $\alpha_t^n=-t(\alpha_t+1)$, and since $\alpha_t\in\mathfrak m_q$ we have
$\alpha_t+1\equiv 1\pmod{\mathfrak m_q}$, hence $\alpha_t+1\in R_q^\times$.
Thus $t\in(\alpha_t)$, and writing $t=qu$ with $u\in\Z_q^\times$ gives $q\in(\alpha_t)$.
Consequently $\mathfrak m_q=(q,\alpha_t)=(\alpha_t)$ is principal.

Since $R_q$ is a $1$-dimensional Noetherian local domain with principal maximal ideal, it is a DVR.
In particular $R_q$ is integrally closed, hence equals the ring of integers of $\Q_q(\alpha_t)$.
Therefore $\Ocal_{K_t}\otimes\Z_q \cong R_q$, so $q\nmid I_t$.

If $p\mid I_t$ then $p^2\mid \disc(f_t)$ by the index--discriminant relation
$\disc(f_t)=\disc(K_t)\,I_t^2$.
If $p\nmid t$, then \eqref{eq:disc-ft} implies $p^2\mid L_n(t)$.
But $tL_n(t)$ is squarefree by hypothesis, hence $L_n(t)$ is squarefree; contradiction.
Thus $p\nmid I_t$ for all primes $p\nmid t$.

Combining above observations yields $I_t=1$, hence $\Ocal_{K_t}=\Z[\alpha_t]$ and $K_t$ is monogenic.

Write
\[
F(t):=t\,L_n(t)=t\,(C_{0,n}+C_{1,n}t)\in\Z[t].
\]
Since $\gcd(C_{0,n},C_{1,n})=1$, the two linear factors $t$ and $L_n(t)$ are coprime in $\Z[t]$,
so $F$ is a squarefree polynomial.

For each prime $\ell$, let
\[
\rho(\ell^2):=\#\{\,a\in\Z/\ell^2\Z:\ \ell^2\mid F(a)\,\}.
\]
A direct congruence count gives:
\begin{enumerate}[label=\textup{(\roman*)}]
\item If $\ell\mid n$ (equivalently $\ell\mid C_{0,n}$), then $C_{1,n}\not\equiv 0\pmod\ell$ and
$L_n(a)\equiv C_{1,n}a\pmod\ell$. Hence $\ell^2\mid aL_n(a)$ iff $\ell\mid a$, so $\rho(\ell^2)=\ell$.
\item If $\ell\mid (n-1)$ (equivalently $\ell\mid C_{1,n}$), then $C_{0,n}\not\equiv 0\pmod\ell$, so
$L_n(a)\not\equiv 0\pmod\ell$ for all $a$, and therefore $\ell^2\mid aL_n(a)$ iff $\ell^2\mid a$.
Thus $\rho(\ell^2)=1$.
\item If $\ell\nmid n(n-1)$, then $C_{0,n}$ and $C_{1,n}$ are both units modulo $\ell^2$ and the two linear
congruences $a\equiv 0\pmod{\ell^2}$ and $L_n(a)\equiv 0\pmod{\ell^2}$ each contribute exactly one solution.
Since they are distinct (because $C_{0,n}\not\equiv 0\pmod\ell$), we have $\rho(\ell^2)=2$.
\end{enumerate}

In particular, for every prime $\ell$ we have $\rho(\ell^2)<\ell^2$, and the Euler product
\[
\mathfrak S:=\prod_{\ell}\Bigl(1-\frac{\rho(\ell^2)}{\ell^2}\Bigr)
\]
converges absolutely to a positive real number (since $\rho(\ell^2)\ll 1$ for all but finitely many $\ell$,
and $\sum_\ell \ell^{-2}<\infty$). By the standard squarefree sieve/inclusion--exclusion for polynomial values
in one variable (applied to $F(t)$), the set
\[
\{t\in\Z:\ F(t)\ \text{is squarefree}\}
\]
has natural density $\mathfrak S>0$. Intersecting with the congruence condition $\gcd(t,n(n-1))=1$
removes only finitely many residue classes modulo $\prod_{\ell\mid n(n-1)}\ell$, hence preserves positivity
of density. Finally, excluding $|t|\le 1$ removes finitely many integers. Therefore $\mathrm{dens}(\mathcal T_n)>0$.

\end{proof}

\subsubsection{A fixed-index twist of the distinguished generator}

Fix an integer $c\ge 2$, and define
\[
F_{t,c}(X)\ :=\ X^n+c^{\,n-1}tX+c^{\,n}t\ \in\ \Z[X],\qquad K_{t,c}:=\Q(\theta_{t,c}),
\]
where $\theta_{t,c}$ is any root of $F_{t,c}$.
If $\alpha_t:=\theta_{t,c}/c$, then $\alpha_t$ satisfies $f_t(\alpha_t)=0$ and hence $K_{t,c}=K_t$.

\begin{theorem}[Fixed index but positive \textsc{ABS}--unobstructed density]
\label{thm:fixed-index-twist}
Fix $n\ge 4$ and $c\ge 2$, and set
\[
\mathcal T_{n}(c)\ :=\ \{t\in \mathcal T_n:\ \gcd(t,c)=1\}.
\]
Then $\mathrm{dens}(\mathcal T_n(c))>0$, and for every $t\in \mathcal T_n(c)$ the field $K_{t,c}$ is \textsc{ABS}--unobstructed while
the distinguished generator $\theta_{t,c}$ satisfies
\[
[\Ocal_{K_{t,c}}:\Z[\theta_{t,c}]]\ =\ c^{N}\ >\ 1,
\qquad N=\frac{n(n-1)}2.
\]
In particular, the one-parameter family $\{K_{t,c}\}_{t\in\Z}$ contains a positive-density set of parameters $t$
for which \emph{(i)} the \textsc{ABS} fixed--sign local condition holds and \emph{(ii)} $\Z[\theta_{t,c}]$ is a proper
suborder of fixed index $c^{N}$.
\end{theorem}

\begin{proof}
Since $\mathcal T_n$ has positive density and $\gcd(t,c)=1$ excludes only finitely many residue classes modulo primes dividing $c$,
we have $\mathrm{dens}(\mathcal T_n(c))>0$.

Fix $t\in\mathcal T_n(c)$ and put $\alpha_t=\theta_{t,c}/c$.
By Theorem~\ref{thm:ft-monogenic-positive-density}, $\Ocal_{K_{t,c}}=\Ocal_{K_t}=\Z[\alpha_t]$; in particular,
$K_{t,c}$ is monogenic and hence \textsc{ABS}--unobstructed.

As $\Z$--modules,
\[
\Z[\alpha_t]\ =\ \Z\langle 1,\alpha_t,\alpha_t^2,\dots,\alpha_t^{n-1}\rangle,
\qquad
\Z[\theta_{t,c}]\ =\ \Z\langle 1,c\alpha_t,c^2\alpha_t^2,\dots,c^{n-1}\alpha_t^{n-1}\rangle.
\]
Relative to the basis $(1,\alpha_t,\dots,\alpha_t^{n-1})$ of $\Z[\alpha_t]$,
the basis $(1,\theta_{t,c},\dots,\theta_{t,c}^{\,n-1})$ is obtained by the diagonal change--of--basis matrix
$\mathrm{diag}(1,c,c^2,\dots,c^{n-1})$, whose determinant is $c^{1+2+\cdots+(n-1)}=c^{N}$.
Hence
\[
[\Ocal_{K_{t,c}}:\Z[\theta_{t,c}]]=[\Z[\alpha_t]:\Z[\theta_{t,c}]]=c^{N}>1.
\]
\end{proof}

\begin{corollary}[Concrete quartic instance]
\label{cor:quartic-fixed-index}
Take $n=4$ and $c=2$.  Then for every $t\in\mathcal T_4(2)=\mathcal T_4$ we have
\[
F_{t,2}(X)=X^4+8tX+16t,\qquad
[\Ocal_{K_{t,2}}:\Z[\theta_{t,2}]]=2^{6}=64,
\]
and $K_{t,2}$ is \textsc{ABS}--unobstructed.  Moreover,
\[
\disc(X^4+tX+t)=t^3(256-27t),\qquad \disc(F_{t,2})=2^{12}\,t^3(256-27t).
\]
\end{corollary}

The fixed-index twist produces a positive-density set of \textsc{ABS}--unobstructed fields in which the \emph{distinguished}
generator has fixed index $>1$, so any statement of the form
\textsc{ABS}--unobstructed $\Rightarrow$ distinguished generator has index $1$ fails in general one-parameter families.

\begin{remark}
To conclude, we list other classes of families beyond pure fields and literal Eisenstein primes that our method can apply.
\begin{enumerate}

\item \textbf{Kummer families attached to polynomial values.}
Replace $x^n-m$ by $x^n-F(\mathbf{t})$ with $\mathbf{t}\in\mathbb{Z}^r$.
Primes $\mathfrak{q}\mid F(\mathbf{t})$ with valuation $1$ are portable in the same way as in the pure case,
and for $r\ge 2$ one can combine portability with geometric density tools for squarefree values.
Modern work on squarefree discriminants and maximality in large families suggests many opportunities
to embed rigid portable-prime subfamilies inside high-dimensional spaces.

\item \textbf{Multi-parameter Eisenstein divisors and higher-dimensional sieves} The present closure step is essentially one-dimensional (sieving along a squarefree or almost squarefree
parameter). In genuinely higher-dimensional coefficient spaces, the appropriate replacement is the
Ekedahl--Bhargava \emph{geometric sieve}, which is designed to control integral points subject to
congruence conditions that vary with the prime. A schematic model is the two-parameter family
\[
(a,b)\in\mathbb{Z}^2,\qquad
f_{a,b}(x)=x^n + a\,h(x) + b,
\]
with fixed $h\in\mathbb{Z}[x]$. For many choices of $h$, the condition that $f_{a,b}$ be Eisenstein at a prime $q$
can be expressed by explicit congruences (typically modulo $q^2$), e.g.\ of the form
\[
b\equiv 0 \pmod q,\quad b\not\equiv 0 \pmod{q^2},
\]
together with $q$-divisibility constraints on the remaining coefficients. Thus, for each prime $q$, Eisenstein at $q$ cuts out a congruence locus in parameter space---an
\emph{Eisenstein divisor} in the $(a,b)$-plane---and one is led to count parameters meeting one (or many) such prime-indexed loci.

\item \textbf{Newton polygon portability (beyond Eisenstein).}
For many sparse families, local integral-basis and ramification structure can be read from higher Newton polygons. This suggests replacing Eisenstein at $\mathfrak{q}$ by a Newton-polygon portability criterion guaranteeing that the completed order is a DVR generated by an explicit uniformizer.
See \cite{GMN2015} for a framework connecting Newton polygons, discriminants, and integral bases.

\end{enumerate}
\end{remark}

\bibliographystyle{alpha}
\bibliography{References}

\end{document}